\documentclass[11pt,reqno]{amsart}

\usepackage{fullpage}

\usepackage[margin=1in]{geometry}
\usepackage{amsmath}
\usepackage{amsfonts}
\usepackage{amssymb}
\usepackage{amsthm,mathtools,mathabx,accents,}
\usepackage{newlfont}
\usepackage{graphicx}
\usepackage{xcolor,enumerate,mathrsfs, scalerel}
\usepackage{esint}
\usepackage{hyperref}

\newtheorem{thm}{Theorem}[section]
\newtheorem{cor}[thm]{Corollary}
\newtheorem{lem}[thm]{Lemma}
\newtheorem{prop}[thm]{Proposition}

\theoremstyle{definition}
\newtheorem{defn}[thm]{Definition}
\newtheorem{rem}[thm]{Remark}
\numberwithin{equation}{section}

\renewcommand{\div}{\operatorname{div}}

\newcommand\de{\delta}

\newcommand\ve{\varepsilon}

\newcommand{\la}{\lambda}

\newcommand{\supp}{\operatorname{supp}}

\DeclarePairedDelimiter{\ceil}{\lceil}{\rceil}

\usepackage[usestackEOL]{stackengine}
\def\dint{\,\ThisStyle{\ensurestackMath{%
  \stackinset{c}{.2\LMpt}{c}{.5\LMpt}{\SavedStyle-}{\SavedStyle\phantom{\int}}}%
  \setbox0=\hbox{$\SavedStyle\int\,$}\kern-\wd0}\int}

\title{Non-unique solutions for electron MHD}

\begin{document}

\author [Mimi Dai]{Mimi Dai}

\address{Department of Mathematics, Statistics and Computer Science, University of Illinois at Chicago, Chicago, IL 60607, USA}
\email{mdai@uic.edu} 

\thanks{The author is partially supported by the NSF grants DMS--2009422 and DMS--2308208 and the Simons Foundation. }

\begin{abstract}
We consider the electron magnetohydrodynamics (MHD) equation on the 3D torus $\mathbb T^3$. For a given smooth vector field $H$ with zero mean and zero divergence, we can construct a weak solution $B$ to the electron MHD in the space $L^\gamma_tW^{1,p}_x$ for appropriate $(\gamma, p)$ such that $B$ is arbitrarily close to $H$ in this space. The parameters $\gamma$ and $p$ depend on the resistivity. As a consequence, non-uniqueness of weak solutions is obtained for the electron MHD with hyper-resistivity. In particular, non-Leray-Hopf solutions can be constructed. As a byproduct, we also show the existence of weak solutions to the electron MHD without resistivity.

\bigskip


\hspace{0.02cm}CLASSIFICATION CODE: 35Q35, 76B03, 76D09, 76E25, 76W05.
\end{abstract}

\maketitle

\section{Introduction}

Consider the electron magnetohydrodynamics (MHD) equation with a general hyper-resistive term on $\mathbb T^3$
\begin{equation}\label{emhd}
B_t+\nabla\times((\nabla\times B)\times B)=-(- \Delta)^\alpha B
\end{equation}
for $\alpha\geq1$.
Note for initial data $B_0$ satisfying $\nabla\cdot B_0=0$, the solution of \eqref{emhd} preserves the Gauss law, i.e. $\nabla\cdot B=0$ for all the time. When $\alpha=1$, \eqref{emhd} is the physically relevant important model in plasma physics which describes the motion of magnetic field while the background ion flow motion is slow. The nonlinear term in \eqref{emhd} is deemed to capture the rapid magnetic reconnection phenomena due to the Hall effect. For more physics background regarding this model, we refer the reader to the book \cite{Bis1}. Our study of \eqref{emhd} for general $\alpha\geq 1$ stems from purely mathematical interests.

This paper concerns constructing weak solutions for the electron MHD. We start with the definition of a weak solution.

\begin{defn} Let $\mathcal D_T$ be the space of test functions $\varphi\in C^\infty(\mathbb T^3\times \mathbb R)$ satisfying $\div \varphi=0$ and $\varphi=0$ for $t\geq T$. Given $B_0\in L^2(\mathbb T^3)$ with $\div B_0=0$ in the weak sense, a vector field $B\in L^2(\mathbb T^3\times [0,T])$ is said to be a weak solution of (\ref{emhd}) with initial data $B_0$ if $B(t)$ is weakly divergence-free for a.e. $t\in[0,T]$, and 
\begin{equation}\label{weak}
\int_{\mathbb T^3} B_0(x)\cdot \varphi(x,0)\,dx=-\int_0^T\int_{\mathbb T^3}B\cdot \left(\partial_t\varphi-(-\Delta)^\alpha \varphi+B\cdot\nabla(\nabla\times \varphi) \right)dxdt
\end{equation}
for any $\varphi\in \mathcal D_T$.
\end{defn}
In the weak formulation \eqref{weak} we used the fact that 
\[\nabla\times((\nabla\times B)\times B)=\nabla\times \div (B\otimes B) \]
for divergence free vector field $B$.

Formally we have the energy identity for (\ref{emhd}) 
\begin{equation}\notag
\|B(t)\|_{L^2}^2+2\int_0^t \|(-\Delta)^{\frac\alpha2} B(s)\|_{L^2}^2ds= \|B(0)\|_{L^2}^2
\end{equation}
by noticing the cancellation 
\[\int_{\mathbb T^3} \nabla\times((\nabla\times B)\times B)\cdot B dx=\int_{\mathbb T^3} ((\nabla\times B)\times B)\cdot (\nabla\times B) dx=0\]
for a smooth vector field $B$. Thus we can adapt the notion of Leray-Hopf solution for (\ref{emhd}) as follows.

\begin{defn}
A weak solution $B$ of (\ref{emhd}) is called Leray-Hopf type of weak solution if $B\in C_w([0,T]; L^2(\mathbb T^3))\cap L^2(0,T; H^\alpha(\mathbb T^3))$ and the energy inequality 
\begin{equation}\notag
\|B(t)\|_{L^2}^2+2\int_0^t \|(-\Delta)^{\frac\alpha2} B(s)\|_{L^2}^2ds\leq \|B(0)\|_{L^2}^2
\end{equation}
holds for $t\in[0,T]$.
\end{defn}
In the physically relevant case $\alpha=1$, the Leray-Hopf space for \eqref{emhd} is 
\[C_w([0,T]; L^2(\mathbb T^3))\cap L^2(0,T; H^1(\mathbb T^3))\]
which is the same as for the Navier-Stokes equation.

The electron MHD has the natural scaling in the sense that if $B(x,t)$ is a solution with initial data $B_0(x)$, the rescaled magnetic field 
\[B_\lambda=\lambda^{2\alpha-2}B(\lambda x,\lambda^{2\alpha}t)\]
is also a solution associated with the rescaled initial data $B_0(\lambda x)$. Invariant functional spaces under such scaling include the Sobolev spaces $L_t^rW_x^{1,s}$ with $r,s$ satisfying 
\begin{equation}\notag
\frac{2\alpha}r+\frac3s= 2\alpha-1, \ \ \ 1\leq r, s\leq \infty. 
\end{equation}
The two ending point critical cases of $r=\infty$ and $s=\infty$ are respectively
\begin{equation}\label{critical-space}
L^\infty_t W_x^{1, \frac3{2\alpha-1}}, \ \ \ L^{\frac{2\alpha}{2\alpha-1}}_t W_x^{1, \infty}.
\end{equation}
It was shown in \cite{CL, Ye} that a solution to (\ref{emhd}) in the space $L_t^rW_x^{1,s}$ with the Ladyzhenskaya-Prodi-Serrin (LPS) type of condition
\begin{equation}\label{LPS}
\frac{2\alpha}r+\frac3s\leq 2\alpha-1, \ \ \mbox{for} \ \  s\in(3, \infty]
\end{equation}
is regular and hence unique. We note that compared to the LPS spaces $L_t^rW_x^{1,s}$ with $r,s$ satisfying \eqref{LPS}, the Leray-Hopf space $L_t^\infty L_x^2\cap L_t^2 H_x^\alpha$ has much lower regularity.

We shall adapt the convex integration method to construct weak solutions for the electron MHD. As a consequence, non-uniqueness of weak solutions is obtained in the hyper-resistive case.

\medskip

\subsection{Main results}
\label{sec-results}

\begin{thm}\label{thm-main}
Let $H$ be a smooth zero-mean vector field on $\mathbb T^3\times [0,T]$ with $\div H=0$. Then for any $\ve_*>0$, there exists a weak solution $B$ of (\ref{emhd}) with initial data $H(x,0)$ and spatial zero-mean such that 
\begin{itemize}
\item [(i)] in the case of $\alpha\in[1,2]$, 
\[B\in L^\gamma_t W^{1,\infty}_x, \ \ \gamma<\frac43, \ \ \mbox{and}\]
\[\|B-H\|_{L^1_tL^2_x}+\|B-H\|_{L^\gamma_tW^{1,\infty}_x}\leq \ve_*;\]
\item [(ii)] in the case of $\alpha\in[1,3)$, 
\[
B\in L^\infty_t W^{1,p}_x \ \ \ \mbox{for} \ \ 
\begin{cases}
p<\frac65, \ \ \alpha\in[1, \frac74), \\
p<\frac{3}{2\alpha-1}, \ \ \alpha\in[\frac74,3)
\end{cases}
\]
and
\[\|B-H\|_{L^1_tL^2_x}+\|B-H\|_{L^\infty_tW^{1,p}_x}\leq \ve_*.\]
\end{itemize}

\end{thm}

Non-uniqueness of weak solutions can be derived from this theorem in the case $\alpha\geq \frac74$, which corresponds to the critical and subcritical regime of the equation.

\begin{cor}\label{cor-1}
Let $\alpha\in[\frac74,2]$ and $\gamma<\frac43$. The following statements hold:
\begin{itemize}
\item [(i)]
For any weak solution $\widetilde B$ of (\ref{emhd}), there exists another weak solution $B\in L^\gamma_t W^{1,\infty}_x$ of (\ref{emhd}) with the same initial data. 
\item [(ii)]
There exist a weak solution $B\in L^\gamma_t W^{1,\infty}_x$ which is not a Leray-Hopf solution.
\item [(iii)]
For every divergence-free initial data in $L^2_{x}$, there are infinitely many weak solutions in $B\in L^\gamma_t W^{1,\infty}_x$.
\end{itemize}
\end{cor}

\begin{cor}\label{cor-2}
Let $\alpha\in[\frac74, 3)$ and $p<\frac{3}{2\alpha-1}$. The following statements hold:
\begin{itemize}
\item [(i)]
For any weak solution $\widetilde B$ of (\ref{emhd}), there exists another weak solution $B\in L^\infty_t W^{1,p}_x$ of (\ref{emhd}) with the same initial data.
\item [(ii)]
There exist a weak solution $B\in L^\gamma_t W^{1,\infty}_x$ which is not a Leray-Hopf solution.
\item [(iii)] 
For every divergence-free initial data in $L^2_{x}$, there are infinitely many weak solutions in $B\in L^\infty_t W^{1,p}_x$.
\end{itemize}
\end{cor}

\begin{rem}
The result of Corollary \ref{cor-1} is sharp when $\alpha=2$, since the critical scaling $\frac{2\alpha}{2\alpha-1}=\frac43$ and the weak solution $B$ belongs to $L^\gamma_t W^{1,\infty}_x$ for $\gamma<\frac43$. The result of Corollary \ref{cor-2} is sharp for all $\alpha\in[\frac74, 3)$ in view of \eqref{critical-space}.
\end{rem}

\begin{rem}
Following the techniques of \cite{BCV, ChL}, the solutions in Theorem \ref{thm-main}, Corollary \ref{cor-1} and Corollary \ref{cor-2} can be made smooth almost everywhere in time. As done in \cite{Gor, LQZZ}, one can also obtain non-uniqueness for the electron MHD \eqref{emhd} in spaces $L^\gamma_tW^{1,p}_x$ for a larger range of $\gamma, p$ than what is covered in Theorem \ref{thm-main}. 
We do not pursue in these directions in the current paper.
\end{rem}

The electron MHD \eqref{emhd} without the term $(-\Delta)^\alpha B$ is referred the ideal or non-resistive electron MHD. Since the Hall term in \eqref{emhd} is more singular than the nonlinear term of the Euler equation, it was known to be challenging to show existence of weak solutions for the ideal electron MHD in general. Nevertheless, our construction of weak solutions through the convex integration approach does not depend on the resistivity. Therefore we are able to construct weak solutions for the ideal electron MHD.

\begin{thm}\label{thm-2}
Consider the non-resistive electron MHD,
\begin{equation}\notag
\begin{split}
B_t+\nabla\times((\nabla\times B)\times B)=0,\\
\nabla\cdot B=0.
\end{split}
\end{equation}
Let $H$ be a smooth zero-mean vector field on $\mathbb T^3\times [0,T]$ with $\div H=0$. Then for any $\ve_*>0$, there exists a weak solution $B$ of the non-resistive electron MHD with initial data $H(x,0)$ and spatial zero-mean such that either one of the following statements holds:
\begin{itemize}
\item [(i)] 
\[B\in L^\gamma_t W^{1,\infty}_x, \ \ \gamma<\frac43, \ \ \mbox{and}\]
\[\|B-H\|_{L^1_tL^2_x}+\|B-H\|_{L^\gamma_tW^{1,\infty}_x}\leq \ve_*;\]
\item [(ii)] 
\[
B\in L^\infty_t W^{1,p}_x \ \ \ p<\frac65, \ \ \mbox{and}
\]
\[\|B-H\|_{L^1_tL^2_x}+\|B-H\|_{L^\infty_tW^{1,p}_x}\leq \ve_*.\]
\end{itemize}
\end{thm}

\medskip

\subsection{Relevant previous work}
The unique solvability for supercritical equations is a challenging problem in general. The convex integration approach has been proven a robust machinery to construct non-unique weak solutions. In the case of pure hydrodynamics, Buckmaster and Vicol \cite{BV} first showed non-unique weak solutions for the Navier-Stokes equation (NSE) in $C^0_tH_x^\beta$ for $0<\beta\ll 1$. The regularity of the solutions has a certain gap with the Leray-Hopf space and the Ladyzhenskaya-Prodi-Serrin space for the NSE. Taking advantage of the temporal intermittency, Cheskidov and Luo \cite{ChL2, ChL} were able to construct weak solutions for the NSE in spaces that touch the LPS borderline. Inspired by the ideas of \cite{ChL}, Giri and Radu \cite{GR} designed a convex integration scheme involving two steps of (temporal and spatial) perturbations and resolved the Onsager conjecture for the 2D Euler equation. Li, Qu, Zeng and Zheng \cite{LQZZ} extended the techniques of \cite{ChL2, ChL} to hyper-dissipative NSE where non-uniqueness was obtained in a class of Sobolev spaces. The result of \cite{LQZZ} was further optimized in \cite{Gor} by Gorini. Non-uniqueness in spaces near the LPS line was obtained for the classical MHD by Li, Zeng and Zheng \cite{LZZ}. For the full MHD with Hall effect, the author \cite{Dai} constructed non-unique weak solutions in $C^0_tL^2_x\cap L^2_tH^1_x$ (the space of Leray-Hopf), although the solutions are not known to satisfy the energy inequality. 

\medskip

The rest of the paper is devoted to proving the main results stated in Subsection \ref{sec-results}.

\bigskip

\section{The main inductive proposition and heuristic analysis}

\subsection{The relaxed system}

Let $A$ be the zero-mean magnetic vector potential satisfying $B=\nabla\times A$ and the Coulomb gauge condition $\nabla\cdot A=0$.  
Note that 
\[\nabla\times((\nabla\times B)\times B)=\nabla\times
\nabla\cdot (B\otimes B).\]
The electron MHD equation (\ref{emhd}) can thus be recasted in term of $A$ as
\begin{equation}\label{emhd-a}
\begin{split}
A_t+\nabla\cdot (B\otimes B)+\nabla P=& -(-\Delta)^\alpha A,\\
B=&\ \nabla\times A,\\
\nabla\cdot A=&\ 0
\end{split}
\end{equation}
for a pressure function $P$. 

We consider the relaxed system of (\ref{emhd-a}) in the form
\begin{equation}\label{emhd-q}
\begin{split}
\partial_t A+\nabla\cdot(B\otimes B)+\nabla P=& -(-\Delta)^\alpha A+ \nabla\cdot \mathring R,\\
B=&\ \nabla\times A,\\
\nabla\cdot A=&\ 0
\end{split}
\end{equation}
where $\mathring R$ is a symmetric traceless tensor.
Exploiting the convex integration scheme, we construct solutions iteratively to the approximating systems
\begin{equation}\label{emhd-q}
\begin{split}
\partial_t A_q+\nabla\cdot (B_q\otimes B_q)+\nabla P_q=& -(-\Delta)^\alpha A_q+ \nabla\cdot \mathring R_q,\\
B_q=&\ \nabla\times A_q,\\
\nabla\cdot A_q=&\ 0
\end{split}
\end{equation}
for $q\in \mathcal N$. 

\subsection{Iteration statement}
For large constants $a,b>0$, define the spatial frequency at $q$-th level as
\begin{equation}\notag
\la_q = 2\pi \ceil{a^{(b^q)}}, \ \ q\geq 0
\end{equation}
and the amplitude parameter as
\begin{equation}\label{def.deq}
    \de_q = \la_q^{-2\beta}, \ \ q\geq 2
\end{equation}
where $\beta>0$ (rather small) quantifies the regularity of the constructed solutions. The amplitude $\delta_0$ and $\delta_1$ will be chosen appropriately in the construction later on.
The large constants $a$ and $b$ and small $\beta$ are chosen such that $0<b^2\beta<\frac{1}{1000}$. By convention, we adapt the symbol $\lesssim$ to denote an estimate of $\leq$ up to a constant multiple in the rest of the paper.

We shall construct solutions $(A_q,B_q,\mathring R_q)$ of \eqref{emhd-q} iteratively with frequency support near $\lambda_q$ that satisfy the following inductive assumptions, 
\begin{equation}\label{induct-A}
\|A_q\|_{L^\infty_tH^3_x}\lesssim \lambda_q^5,
\end{equation}
\begin{equation}\label{induct-B}
\|B_q\|_{L^\infty_tH^2_x}\lesssim \lambda_q^5,
\end{equation}
\begin{equation}\label{induct-R}
\|R_q\|_{L^1_{t,x}}\leq \delta_{q+1}.
\end{equation}
For a time interval $I\subset [0,T]$ and a constant $\epsilon>0$, denote by $N_\epsilon(I)$ the $\epsilon$-neighborhood of $I$ as
\[N_\epsilon(I)=\{t\in[0,T]: |t-s|\leq \epsilon \ \ \mbox{for some} \ \ s\in I \}.\]

\begin{prop}  [Main iteration] \label{prop}
Let 
\begin{equation}\notag
L^\gamma_t W^{1,p}_x=
\begin{cases}
L^\gamma_t W^{1,\infty}_x, \ \ \ \gamma<\frac43, \ \ \mbox{for} \ \ \alpha\in[1,2],\\
L^\infty_t W^{1,p}_x, \ \ \ p<\frac65, \ \ \mbox{for} \ \ \alpha\in[1, \frac74),\\
L^\infty_t W^{1,p}_x, \ \ \ p<\frac{3}{2\alpha-1}, \ \ \mbox{for} \ \ \alpha\in[\frac74, 3). 
\end{cases}
\end{equation}
There exist large enough constants $a, b>0$, sufficiently small $\beta>0$ and a large constant $M>0$ such that the following statement holds:
Assume $(A_q,B_q, \mathring R_q)$ is a solution of (\ref{emhd-q}) that satisfies (\ref{induct-A})-(\ref{induct-R}). There exists another solution $(A_{q+1},B_{q+1}, \mathring R_{q+1})$ of (\ref{emhd-q}) satisfying (\ref{induct-A})-(\ref{induct-R}) with $q$ replaced by $q+1$. Moreover, we have
\begin{equation}\label{iter-1}
\|B_{q+1}-B_q\|_{L^2_{t,x}}\leq M\delta_{q+1}^{\frac12},
\end{equation}
\begin{equation}\label{iter-2}
\|B_{q+1}-B_q\|_{L^1_t L^2_x}\leq \delta_{q+2}^{\frac12},
\end{equation}
\begin{equation}\label{iter-3}
\|B_{q+1}-B_q\|_{L^\gamma_t W^{1,p}_x}\leq \delta_{q+2}^{\frac12},
\end{equation}
\begin{equation}\label{iter-supp}
\supp_t (A_{q+1},B_{q+1}, \mathring R_{q+1})\subset N_{\delta_{q+2}^{\frac12}}(\supp_t (A_{q},B_{q}, \mathring R_{q})).
\end{equation}
\end{prop}

\subsection{Heuristics}
\label{sec-heuristics}
Denote by $v_{q+1}$ and $w_{q+1}$ the perturbations for $A_q$ and $B_q$ respectively, i.e.
\[A_{q+1}=A_q+v_{q+1}, \ \ \ B_{q+1}=B_q+w_{q+1}.\]
The triplet $(A_{q+1}, B_{q+1}, R_{q+1})$ with a new stress error $R_{q+1}$ solves (\ref{emhd-q}) at the $(q+1)$-th level. Straightforward algebra shows that the new stress error satisfies
\begin{equation}\label{heu-error}
\begin{split}
\nabla\cdot R_{q+1}=&\ \partial_tv_{q+1} +(-\Delta)^\alpha v_{q+1}+\nabla\cdot(B_q\otimes w_{q+1}+w_{q+1}\otimes B_q)\\
&+\nabla\cdot(w_{q+1}\otimes w_{q+1}+R_q)+\nabla P_{q+1}.
\end{split}
\end{equation}
The terms in the first line of the right hand side of \eqref{heu-error} are called linear errors, while the first term on the second line will produce oscillation error. The purpose of the perturbation $w_{q+1}$ is to reduce the previous error term $R_q$ such that the resulted oscillation error from 
\[\nabla\cdot(w_{q+1}\otimes w_{q+1}+R_q)\]
is very small. Through the iteration, we expect to have the limit $R_q\to 0$ in $L^1_t L_x^1$ and hence $w_q\to 0$ in $L^2_t L_x^2$ as $q\to \infty$.
Thus, in the same time, we need to control the linear errors in the space $L^1_t L_x^1$ (in order to pass to a limit in the weak formulation as $q\to\infty$). More specifically, we expect to have
\begin{equation}\label{heu-est}
\begin{split}
\|\mathcal R\partial_tv_{q+1} \|_{L^1_t L^1_x}\ll 1,\\
\|\mathcal R(-\Delta)^\alpha v_{q+1} \|_{L^1_t L^1_x}\ll 1,\\
\|\mathcal R \div (B_q\otimes w_{q+1}+w_{q+1}\otimes B_q) \|_{L^1_t L^1_x}\ll 1
\end{split}
\end{equation}
where $\mathcal R$ denotes the inverse operator of $\div$. Without diving into detail, we point out that the term $\div(B_q\otimes w_{q+1}+w_{q+1}\otimes B_q)$ is a minor error compared to other linear errors. Hence we focus on the first two inequalities of \eqref{heu-est} in the following analysis. 

To have the limit solution in the aimed space $L^\gamma_t W^{1,p}_x$, we naturally impose
\begin{equation}\label{heu-space}
\|w_{q+1} \|_{L^\gamma_t W^{1,p}_x}\ll 1.
\end{equation}

Denote by large constants $\lambda$ and $\tau$ the spatial oscillation frequency and temporal oscillation frequency respectively of the perturbations $w_{q+1}$ and $v_{q+1}$. We assume a full dimension of temporal concentration and $(3-D)$ dimension of spatial concentration, with $D$ being the spatial intermittency dimension (cf. \cite{CD-nse} for a mathematical definition of intermittency dimension). Note the temporal intermittency dimension is 0. Under such setting, the scaling analysis shows
\begin{equation}\notag
\begin{split}
\|\mathcal R\partial_tv_{q+1} \|_{L^1_t L^1_x}&\sim \lambda^{-1}\tau \|v_{q+1} \|_{L^1_t L^1_x}\sim \lambda^{-2}\tau \|w_{q+1} \|_{L^1_t L^1_x}
\sim \lambda^{-2}\tau \tau^{-\frac12}\lambda^{-\frac12(3-D)} \|w_{q+1} \|_{L^2_t L^2_x},\\
\|\mathcal R(-\Delta)^\alpha v_{q+1} \|_{L^1_t L^1_x}&\sim \lambda^{2\alpha-2}\|w_{q+1} \|_{L^1_t L^1_x}\sim \lambda^{2\alpha-2}\tau^{-\frac12}\lambda^{-\frac12(3-D)}\|w_{q+1} \|_{L^2_t L^2_x},\\
\|w_{q+1} \|_{L^\gamma_t W^{1,p}_x}&\sim \tau^{\frac12-\frac1\gamma}\lambda^{(\frac12-\frac1p)(3-D)+1} \|w_{q+1} \|_{L^2_t L^2_x}.
\end{split}
\end{equation}
Thus to fulfill \eqref{heu-est} and \eqref{heu-space}, we should have
\begin{equation}\label{heu-para-1}
\begin{split}
\tau^{\frac12}\lambda^{-2-\frac12(3-D)}\lesssim 1,\\
\tau^{-\frac12}\lambda^{2\alpha-2-\frac12(3-D)}\lesssim 1,\\
\tau^{\frac12-\frac1\gamma}\lambda^{(\frac12-\frac1p)(3-D)+1}\lesssim 1.
\end{split}
\end{equation}
Let $\tau=\lambda^n$ for some $n>0$. The conditions of \eqref{heu-para-1} become
\begin{equation}\notag
\begin{split}
\lambda^{\frac12n-2-\frac12(3-D)}\lesssim 1,\\
\lambda^{-\frac12n+2\alpha-2-\frac12(3-D)}\lesssim 1,\\
\lambda^{(\frac12-\frac1\gamma)n+(\frac12-\frac1p)(3-D)+1}\lesssim 1
\end{split}
\end{equation}
which are satisfied provided
\begin{equation}\label{heu-para-2}
\begin{split}
\frac12n-2-\frac12(3-D)\leq 0,\\
-\frac12n+2\alpha-2-\frac12(3-D)\leq 0,\\
(\frac12-\frac1\gamma)n+(\frac12-\frac1p)(3-D)+1\leq 0.
\end{split}
\end{equation}
The first two inequalities of \eqref{heu-para-2} imply
\begin{equation}\label{cond-n}
4\alpha -7 +D\leq n\leq 7-D.
\end{equation}

Starting from here, we distinguish the discussions of the two ending point cases $L_t^\gamma W_x^{1,\infty}$ and $L_t^\infty W_x^{1,p}$.

{\it Case $L_t^\gamma W_x^{1,\infty}$.}  Letting $p=\infty$ in the last inequality of \eqref{heu-para-2} gives
\begin{equation}\notag
\gamma\leq \frac{2}{1+\frac{5-D}n}.
\end{equation}
Combining with \eqref{cond-n}, we have
\begin{equation}\notag
\gamma\leq \frac{2}{1+\frac{5-D}{7-D}}
\end{equation}
which indicates $\gamma$ approaches the possible maximum value $\frac43$ when the spatial intermittency dimension $D$ approaches 3, corresponding to no spatial concentration. 

Following \eqref{cond-n}, we also get $\alpha\leq \frac72-\frac12D\leq 2$ if $D=3$. 

In this ending point case of $L_t^pW_x^{1,\infty}$, the time integrability is low. Hence the solution is rough in time and (relatively) regular in space variable. In another words, the solution is more intermittent in time and less intermittent in space. Thus it is viable to use Mikado flows as building blocks. This guides our choice of building blocks in Subsection \ref{sec-blocks-1}.

{\it Case $L_t^\infty W_x^{1,p}$.} Taking $\gamma=\infty$ in the last inequality of \eqref{heu-para-2} we obtain
\begin{equation}\notag
p<\frac{2}{1+\frac{n+2}{3-D}}
\end{equation}
which implies that we need to choose a minimal intermittency dimension $D\geq 0$ to maximize $p$. Note when $D=0$, it follows from \eqref{cond-n} that $n\geq 4\alpha-7$. Since $n\geq 0$, we have the two subcases 
\begin{equation}\notag
\begin{cases}
p<\frac{2}{1+\frac{D+2}{3-D}}=\frac{2(3-D)}{5}, \ \ \alpha\in[1,\frac74),\\
p<\frac{2}{1+\frac{4\alpha-7+D+2}{3-D}}=\frac{3-D}{2\alpha-1}, \ \ \alpha\in[\frac74, \frac72-\frac{D}2).
\end{cases}
\end{equation}
The direct consequence is that, for $D=0$,
\begin{equation}\notag
\begin{cases}
p<\frac65, \ \ \alpha\in[1,\frac74),\\
p<\frac{3}{2\alpha-1}, \ \ \alpha\in[\frac74, \frac72).
\end{cases}
\end{equation}

Therefore, for the ending point case of $L_t^\infty W_x^{1,p}$ where the spatial integrability (of $\nabla B$) is low, the spatial intermittency $D$ of the solution should be close to 0 in order to maximize the integrability $p$. For this purpose, the classical Mikado flows are not ideal to serve as building blocks. Instead, intermittent jets (first introduced in \cite{BCV}) are more suitable in terms of achieving the desirable intermittency dimension for the building blocks. The basic building blocks of intermittent jets are pipe flows with possibly different length scales in different directions. The stationary pipe flows are not stationary solutions of the electron MHD (similar in the situation of the Euler equation). However, the evolutionary pipe flows moving in a particular direction are approximate solutions of the electron MHD. Detailed construction of the intermittent jets is introduced in Subsection \ref{sec-blocks-2}.

\bigskip

\section{Proof of main results}

\subsection{Proof of Theorem \ref{thm-main}}
\label{sec-proof-main}

At the initial step, take $B_0=\widetilde B$, and the magnetic vector potential $A_0$ with zero-mean, $\nabla\times A_0=B_0$ and $\nabla\cdot A_0=0$. We define 
\begin{equation}\notag
\mathring R_0=\mathcal R(\partial_t A_0+(-\Delta)^\alpha A_0+\div(B_0\mathring \otimes B_0)), \ \ P_0=-\frac13|B_0|^2.
\end{equation}
Since $A_0$ is zero-mean, $(A_0,B_0,R_0)$ satisfies \eqref{emhd-q} at the initial level. Choosing $a,b$ large enough and suitable $\delta_0$, $\delta_1$ can guarantee that \eqref{induct-A}, \eqref{induct-B} and \eqref{induct-R} are satisfied for $q=0$. 

Applying Proposition \ref{prop} iteratively produces a sequence of approximating solutions $(A_q,B_q,R_q)$ satisfying \eqref{induct-A}-\eqref{iter-supp}. In particular, it follows that $\{B_q\}$ is a Cauchy sequence in $L^2_{t,x}\cap L^1_tL^2_x \cap L^\gamma_tW^{1,p}_x$.  Therefore, the sequence has a limit vector field $B\in L^2_{t,x}\cap L^1_tL^2_x \cap L^\gamma_tW^{1,p}_x$, i.e.
\[B_q\to B \ \ \mbox{as} \ \ q\to \infty.\]
In view of the time cut-off function $g\in C^\infty_c([0,T])$ used in the building blocks in Subsection \ref{sec-blocks-1}, it is clear that 
all the perturbations $w_{q+1}$ vanish at time $t=0$ which implies $B_q(x,0)=\widetilde B(x,0)$ for all $q\geq 0$. Hence $B(x,0)=\widetilde B(x,0)$.

We claim that $B$ is a weak solution of \eqref{emhd}. Indeed, for any test function $\varphi\in \mathcal D_T$ we have
\begin{equation}\label{weak-q}
\begin{split}
\int_{\mathbb T^3} B_q(x,0)\cdot \varphi(x,0)\,dx=&-\int_0^T\int_{\mathbb T^3}B_q\cdot \left(\partial_t\varphi-(-\Delta)^\alpha \varphi+B_q\cdot\nabla(\nabla\times \varphi) \right)dxdt\\
&-\int_0^T\int_{\mathbb T^3} R_q: \nabla\varphi \, dxdt.
\end{split}
\end{equation}
Thanks to \eqref{induct-R}, we know $R_q\to 0$ in $L^1_{t,x}$ as $q\to \infty$. Hence 
\[\int_0^T\int_{\mathbb T^3} R_q: \nabla\varphi \, dxdt \to 0 \ \ \ \mbox{as} \ \ q\to \infty.\]
The convergence $B_q\to B$ in $L^2_{t,x}$ guarantees 
\[
\begin{split}
&\int_0^T\int_{\mathbb T^3}B_q\cdot \left(\partial_t\varphi-(-\Delta)^\alpha \varphi+B_q\cdot\nabla(\nabla\times \varphi) \right)dxdt\\
\to & \int_0^T\int_{\mathbb T^3}B\cdot \left(\partial_t\varphi-(-\Delta)^\alpha \varphi+B\cdot\nabla(\nabla\times \varphi) \right)dxdt, \ \ q\to\infty.
\end{split}
\]
Taking $q\to \infty$ in \eqref{weak-q} shows that $B$ is a weak solution of \eqref{emhd}.

In the end, we show that $B$ is close to $\widetilde B$ in $L^1_tL^2_x\cap L^\gamma_tW_x^{1,p}$. Applying \eqref{iter-2} and \eqref{iter-3} yields
\begin{equation}\notag
\begin{split}
&\|B-\widetilde B\|_{L^1_tL^2_x}+\|B-\widetilde B\|_{L^\gamma_tW^{1,p}_x}\\
\leq&\sum_{q=0}^\infty \left(\|B_{q+1}-B_q\|_{L^1_tL^2_x}+\|B_{q+1}-B_q\|_{L^\gamma_tW^{1,p}_x}\right)\\
\leq&\sum_{q=0}^\infty\delta_{q+2}^{\frac12}
\leq \ve_*.
\end{split}
\end{equation}
It completes the proof of Theorem \ref{thm-main}. 

\medskip

\subsection{Proof of Corollary \ref{cor-1}} 
\label{sec-proof-cor1}
We first show the item (i) and item (ii).
As in the assumption, $\widetilde B$ is a weak solution of \eqref{emhd}. If $\widetilde B$ is not a Leray-Hopf solution, there exists a Leray-Hopf solution $B$ to \eqref{emhd} with initial data $B(x,0)=\widetilde B(x,0)$. Thus $B$ is a different solution. Since the equation \eqref{emhd} is critical for $\alpha=\frac74$ and subcritical for $\alpha>\frac74$, the solution $B$ is smooth and unique in $L^\gamma_tW^{1,\infty}_x$.

If $\widetilde B$ is a Leray-Hopf solution, due to the aforementioned reason, $\widetilde B$ is smooth on $(0,T]$. Let $H: [\frac12T, T]\to \mathbb T^3\to \mathbb R^3$ be a smooth zero-mean vector field with $\div H=0$. Moreover, 
\begin{equation}\label{B-bar}
\begin{split}
\widetilde B\equiv H, \ \ \mbox{on} \ \ [\frac12T,\frac34T],\\
\|\widetilde B-H\|_{L^\gamma(\frac12T, T; W^{1,\infty})}\geq 1.
\end{split}
\end{equation} 
For the smooth vector field $H$, applying Theorem \ref{thm-main} gives a weak solution $\bar B$ to \eqref{emhd} on $[\frac12T, T]\times \mathbb T^3$ such that 
\begin{equation}\label{difference-1}
\|\bar B-H\|_{L^\gamma(\frac12T, T; W^{1,\infty})}\leq \ve_*<1.
\end{equation}
We define the vector field $B$ on $[0,T]$ as 
\begin{equation}\notag
B=
\begin{cases}
\widetilde B, \ \ t\in [0,\frac12T]\\
\bar B, \ \ t\in [\frac12T, T]
\end{cases}
\end{equation}
which is apparently a weak solution of \eqref{emhd}, since $\bar B$ and $H$ coincide on $[\frac12T, \frac12T+\epsilon]$ for a small constant $\epsilon>0$.
Appealing to \eqref{B-bar} and \eqref{difference-1} we have
\begin{equation}\notag
\|\widetilde B-B\|_{L^\gamma(\frac12T, T; W^{1,\infty})}\geq \|\widetilde B-H\|_{L^\gamma(\frac12T, T; W^{1,\infty})}-\|H-B\|_{L^\gamma(\frac12T, T; W^{1,\infty})}\geq 1-\ve_*>0.
\end{equation}
Hence the weak solution $B$ is distinct from the Leray-Hopf solution $\widetilde B$. This justifies the first and second conclusions of Corollary \ref{cor-1}.

Regarding conclusion (iii), let $B_0\in L^2_x$ be an initial data with $\div B_0=0$. There is a Leray-Hopf solution $\widetilde B$ to \eqref{emhd} with the initial data $B_0$. Again, $\widetilde B$ is smooth on $(0,T]$. A similar analysis as above can give infinitely many weak solutions of \eqref{emhd} with the same initial data $B_0$. Indeed, for any $j\in \mathbb N$, let $H_j: [0, T]\to \mathbb T^3\to \mathbb R^3$ be a smooth zero-mean vector field with $\div H_j=0$ such that
\begin{equation}\label{Hj}
H_j=
\begin{cases}
\widetilde B, \ \ \ \ \ \mbox{on} \ \ [0,\frac14T],\\
\widetilde B+j\frac{\widetilde B}{\|\widetilde B\|_{L^\gamma(\frac12T, T; W^{1,\infty})}}, \ \ \mbox{on} \ \ [\frac12T,T].
\end{cases}
\end{equation} 
Applying Theorem \ref{thm-main} for each $H_j$, we obtain a weak solution $B_j$ to \eqref{emhd} on $[0,T]$ satisfying 
\begin{equation}\notag
\|B_j-H_j\|_{L^\gamma(0, T; W^{1,\infty})}\leq \ve_*<1.
\end{equation}
It is easy to see $B_j\not\equiv \widetilde B$ on $[0,T]$, since 
\begin{equation}\notag
\|\widetilde B-B_j\|_{L^\gamma(\frac12T, T; W^{1,\infty})}\geq \|\widetilde B-H_j\|_{L^\gamma(\frac12T, T; W^{1,\infty})}-\|H_j-B_j\|_{L^\gamma(\frac12T, T; W^{1,\infty})}\geq j-\ve_*>0.
\end{equation}
Analogously, one notices that $B_j\not\equiv B_{j'}$ for $j\neq j'$. We finish the proof of the second conclusion of the corollary.

\medskip

\subsection{Proof of Corollary \ref{cor-2}}
A minor modification of the analysis of Subsection \ref{sec-proof-cor1} can provide a proof for Corollary \ref{cor-2}.

\medskip

\subsection{Proof of Theorem \ref{thm-2}}
Theorem \ref{thm-2} can be proved exactly the same way as in Subsection \ref{sec-proof-main}. The only modification is that, when applying Proposition \ref{prop}, we take into account the Remark \ref{rem-case1} and Remark \ref{rem-case2}.

\bigskip

\section{Proof of the iteration proposition of case I: boarder line space $L_t^\gamma W_x^{1,\infty}$}
\label{sec-case-1}

Given a solution $(A_q,B_q,R_q)$ of \eqref{emhd-q}, we need to construct another solution $(A_{q+1},B_{q+1},R_{q+1})$ such that the estimates \eqref{induct-A}-\eqref{induct-R} are satisfied with $q$ replaced by $q+1$, and \eqref{iter-1}-\eqref{iter-3} hold as well. The crucial point is to construct appropriate perturbations for $A_q$ and $B_q$ as in all the convex integration schemes in the literature. The heuristic analysis in Subsection \ref{sec-heuristics} provides conceptual guidelines in the construction of the perturbations.

\subsection{Building blocks} 
\label{sec-blocks-1}
As discussed in the heuristic analysis in Subsection \ref{sec-heuristics}, the concentrated Mikado flows (c.f. \cite{ChL, DSz}) serve as our spatial building blocks in this case. 
We choose the spatial concentration parameter $r=\lambda_{q+1}^{n_1}$ for a constant $n_1<0$ to be determined later. The concentration will occur in a 2D plane orthogonal to the Mikado flow. Thus we choose $\Phi:\mathbb R^2\to \mathbb R$ to be a smooth cut-off function with support on the ball $B_1(0)$ satisfying 
\begin{equation}\notag
\phi=-\Delta\Phi, \ \ \ \frac{1}{4\pi^2}\int_{\mathbb R^2}\phi^2(x)\, dx=1.
\end{equation}
We periodize the rescaled functions 
\[\phi_r(x)=r^{-1}\phi(\frac{x}{r}), \ \ \ \Phi_r(x)=r^{-1}\Phi(\frac{x}{r})\]
and use the same notations for the periodized functions which are viewed as periodic functions on $\mathbb T^2$. 

Let $\Lambda$ and the orthonormal bases $\{k,k_1,k_2\}$ be from the geometric Lemma \ref{le-geo}.
We denote the integer $N_{\Lambda}\in\mathbb N$ such that
\[\{N_{\Lambda}k, N_{\Lambda}k_1, N_{\Lambda}k_2\}\subset N_{\Lambda}\mathbb S^2\cap \mathbb Z^3\]
and $M$ the geometric constant satisfying 
\[\sum_{k\in\Lambda}\|\gamma_{(k)}\|_{C^4(B_{1/2}(\mathrm{Id}))}\leq M.\]
We are ready to define the concentrated Mikado flows as 
\begin{equation}\notag
W_{(k)}:=\phi_r(\lambda_{q+1}rN_{\Lambda}k_1\cdot (x-p_k), \lambda_{q+1}rN_{\Lambda}k_2\cdot (x-p_k))k, \ \ k\in\Lambda
\end{equation}
where the points $p_k\in \mathbb R^3$ are chosen such that
\[\supp W_{(k)}\cap \supp W_{(k')}\neq \emptyset, \ \ \mbox{if} \ \ k\neq k'. \]
To ease notation, we write
\begin{equation}\notag
\begin{split}
\phi_{(k)}=&\ \phi_r(\lambda_{q+1}rN_{\Lambda}k_1\cdot (x-p_k), \lambda_{q+1}rN_{\Lambda}k_2\cdot (x-p_k)),\\
\Phi_{(k)}=&\ \Phi_r(\lambda_{q+1}rN_{\Lambda}k_1\cdot (x-p_k), \lambda_{q+1}rN_{\Lambda}k_2\cdot (x-p_k)).
\end{split}
\end{equation}
We observe 
\[W_{(k)}=\nabla\times\nabla\times W_{(k)}^c\]
with $W_{(k)}^c=\frac{1}{\lambda^2N_{\Lambda}^2}\Phi_{(k)}k$. We further note 
\[\div W_{(k)}=0, \ \ \ \div(W_{(k)}\otimes W_{(k)})=0.\]

\begin{lem}\label{le-est-phi}
For $N\in\mathbb N$ and $p\in[1,\infty]$ we have
\begin{equation}\notag
\begin{split}
\|\nabla^N\phi_{(k)}\|_{L^p_x}+\|\nabla^N\Phi_{(k)}\|_{L^p_x}\lesssim &\ r^{\frac2p-1}\lambda_{q+1}^N,\\
\|\nabla^NW_{(k)}\|_{L^p_x}+\lambda_{q+1}^2\|\nabla^NW^c_{(k)}\|_{L^p_x}\lesssim &\ r^{\frac2p-1}\lambda_{q+1}^N
\end{split}
\end{equation}
with implicit constants dependent of $N_\Lambda$ and independent of $r$ and $\lambda_{q+1}$.
\end{lem}
We point out that the choice of Mikado flows above is standard and similar to that of the Navier-Stokes equations. Thus we refer the reader to \cite{ChL} for a proof of the lemma.

To take advantage of temporal intermittency, we adapt the temporal building blocks introduced in \cite{ChL}. We choose the temporal concentration parameter $\tau$ and oscillation parameter $\sigma$ as
\[\tau=\lambda_{q+1}^{n_2}, \ \ \ \sigma=\lambda_{q+1}^{2\ve}\]
where $n_2>0$ is to be fixed later and $\ve>0$ is a sufficiently small constant.
Following the construction in \cite{ChL}, we take $g\in C^{\infty}_c([0,T])$ as a cut-off function such that 
\[\fint_0^Tg^2(t)\, dt=1.\]
We periodize the rescaled function 
\[g_\tau(t)=\tau ^{\frac12}g(\tau t)\]
and treat it as a periodic function on $[0,T]$. We also define
\[h_\tau(t)=\int_0^t(g_\tau^2(s)-1)\, ds, \ \ t\in[0,T].\]
We further denote
\[g_{(\tau)}(t)=g_\tau(\sigma t), \ \ h_{(\tau)}(t)=h_\tau(\sigma t).\]
It is easy to verify 
\begin{equation}\label{cancel-g}
\partial_t(\sigma^{-1}h_{(\tau)})=g_{(\tau)}^2-1. 
\end{equation}

\begin{lem}\cite{ChL}\label{le-est-g}
The estimate 
\[\|\partial_t^Mg_{(\tau)}\|_{L^\gamma_t}\lesssim \sigma^M\tau^{M+\frac12-\frac1\gamma}\]
holds with an implicit constant independent of $\tau$ and $\sigma$. The function $h_{(\tau)}$ satisfies
\[\|h_{(\tau)}\|_{C_t}\leq 1.\]

\end{lem}

\subsection{Cutoff functions}
In order to apply the Geometry Lemma \ref{le-geo}, we introduce a cutoff for the stress error $R_q$. Let $\chi:\mathbb R^3\times \mathbb R^3\to\mathbb R^+$ be a smooth function which is increasing with respect to $|x|$ and satisfies
\begin{equation}\notag
\chi(x)=
\begin{cases}
1, \ \ 0\leq |x|\leq 1,\\
|x|, \ \ |x|\geq 2.
\end{cases}
\end{equation}
Define 
\[\rho=2\chi(R_q).\] 
One can verify that 
\[\mathrm{Id}-\frac{R_q}{\rho}\in B_{1/2}(\mathrm{Id}), \ \ \forall \ \ (x,t)\in \mathbb T^3\times [0,T].\]
With the aim to reduce the stress error $R_q$ by invoking Geometric Lemma \ref{le-geo}, we define the amplitude functions
\begin{equation}\label{def-a}
a_{(k)}(x,t)=\rho^{\frac12}(x,t)\gamma_{(k)}\left(\mathrm{Id}-\frac{R_q(x,t)}{\rho(x,t)} \right), \ \ \ k\in \Lambda.
\end{equation}

\begin{lem}\label{le-est-a}
For $k\in\Lambda$ and $N\geq 0$ we have
\begin{equation}\notag
\begin{split}
\|a_{(k)}\|_{L^2_{t,x}}\lesssim&\ \delta_{q+1}^{\frac12},\\
\|a_{(k)}\|_{C^N_{t,x}}\lesssim&\ 1.
\end{split}
\end{equation}
\end{lem}
The proof is trivial by noticing that $\rho$ is smooth and has the scaling of $R_q$.

\subsection{Perturbations of magnetic and potential fields}
We are ready to define the principal perturbation for the vector potential
\begin{equation}\label{def-per-v}
v_{q+1}^{p}=\sum_{k\in\Lambda}\nabla\times (a_{(k)}g_{(\tau)}W_{(k)}^c)
\end{equation}
and the corresponding perturbation for the magnetic field
\begin{equation}\label{def-per-w}
\begin{split}
\nabla\times v_{q+1}^{p}=&\sum_{k\in\Lambda}\nabla\times\nabla\times (a_{(k)}g_{(\tau)}W_{(k)}^c)\\
=& \sum_{k\in\Lambda}a_{(k)}g_{(\tau)}W_{(k)}+\sum_{k\in \Lambda}g_{(\tau)}\left(\nabla a_{(k)}\times (\nabla\times W_{(k)}^c)+\nabla\times (\nabla a_{(k)}\times W_{(k)}^c)\right)\\
=&: w_{q+1}^{p}+w_{q+1}^{c}.
\end{split}
\end{equation}
It is clear that
\[\nabla\cdot v_{q+1}^{p}=0, \ \ \nabla\cdot (w_{q+1}^{p}+w_{q+1}^{c})=0.\]
We also need to include a temporal corrector in the perturbation of the magnetic vector potential as
\begin{equation}\label{def-v-c}
v_{q+1}^{c}=-\sigma^{-1}\sum_{k\in \Lambda}\mathbb P_{H}\mathbb P_{\neq 0}\left(h_{(\tau)}\fint_{\mathbb T^3}W_{(k)}\otimes W_{(k)}\,dx\nabla a_{(k)}^2 \right)
\end{equation}
in order to cancel the high temporal oscillation in the interactions. Here $\mathbb P_H$ denotes the Helmholtz-Leray projection operator,
\[\mathbb P_H(u)=u+\nabla(-\Delta)^{-1}\div u.\]
Obviously we have $\nabla\cdot v_{q+1}^{c}=0$. Moreover, this vector potential corrector does not yield a perturbation in the magnetic field thanks to $\nabla\times v_{q+1}^c=0$. Indeed, the projector $\mathbb P_H$ commutes with $\mathit{curl}$ by noticing
\[\nabla\times (\mathbb P_H(u))=\nabla\times u+\nabla\times \nabla(-\Delta)^{-1}\div u=\nabla\times u\]
and \[\mathbb P_H(\nabla\times u)=\nabla\times u.\]
Hence 
\[\nabla\times \mathbb P_{H}\mathbb P_{\neq 0}\left(h_{(\tau)}\fint_{\mathbb T^3}W_{(k)}\otimes W_{(k)}\,dx\nabla a_{(k)}^2 \right)=\mathbb P_{H}\mathbb P_{\neq 0}\left(h_{(\tau)}\fint_{\mathbb T^3}W_{(k)}\otimes W_{(k)}\,dx\nabla\times\nabla a_{(k)}^2 \right)=0.\]

In the end, we define the total perturbations of the magnetic field and its vector potential as
\begin{equation}\label{def-per-wv}
w_{q+1}=w_{q+1}^{p}+w_{q+1}^{c}, \ \ v_{q+1}=v_{q+1}^{p}+v_{q+1}^{c}.
\end{equation}
Note $w_{q+1}=\nabla\times v_{q+1}$.

\begin{lem}\label{le-est-wv-1}
For $\gamma\in[1,\infty]$, $\eta\in(1,\infty)$ and $0\leq N\leq 8$, we have
\begin{equation}\label{est-w-p}
\|\nabla^N w_{q+1}^p\|_{L_t^\gamma L_x^\eta}\lesssim \lambda_{q+1}^Nr^{\frac2\eta-1}\tau^{\frac12-\frac1\gamma},
\end{equation}
\begin{equation}\label{est-w-c}
\|\nabla^N w_{q+1}^c\|_{L_t^\gamma L_x^\eta}\lesssim \lambda_{q+1}^{N-1}r^{\frac2\eta-1}\tau^{\frac12-\frac1\gamma},
\end{equation}
\begin{equation}\label{est-v-p}
\|\nabla^N v_{q+1}^p\|_{L_t^\gamma L_x^\eta}\lesssim \lambda_{q+1}^{N-1}r^{\frac2\eta-1}\tau^{\frac12-\frac1\gamma},
\end{equation}
\begin{equation}\label{est-v-c}
\|\nabla^N v_{q+1}^c\|_{L_t^\gamma L_x^\eta}\lesssim \sigma^{-1}
\end{equation}
with implicit constants depending only on $N$, $\gamma$ and $\eta$.  Moreover, the estimate
\begin{equation}\label{est-w-2}
\|w_{q+1}^p\|_{L_t^2 L_x^2}+\|w_{q+1}^c\|_{L_t^2 L_x^2}\lesssim \|R_q\|_{L_t^1 L_x^1}^{\frac12}
\end{equation}
holds.
\end{lem}
We postpone the proof of this lemma to Subsection \ref{sec-blocks-2}, as this one is a special case of Lemma \ref{le-est-wv}.

\subsection{New stress tensor}
Define 
\[A_{q+1}=A_q+v_{q+1}, \ \ \ B_{q+1}=B_q+w_{q+1}.\]
Let $R_{q+1}$ be the new stress error such that $(A_{q+1}, B_{q+1}, R_{q+1})$ solves system (\ref{emhd-q}) at the $(q+1)$-th level. Thus we have
\begin{equation}\label{R-q1-initial}
\begin{split}
\nabla\cdot R_{q+1}-\nabla P_{q+1}=&\ \partial_tv_{q+1}^p-\Delta v_{q+1}+\nabla\cdot(B_q\otimes w_{q+1}+w_{q+1}\otimes B_q)\\
&+\nabla\cdot(w_{q+1}^p\otimes w_{q+1}^p+R_q)+\partial_t v_{q+1}^c\\
&+\nabla\cdot(w_{q+1}^c\otimes w_{q+1}+w_{q+1}^p\otimes w_{q+1}^c)\\
=&:\nabla\cdot R_{\mathrm{lin}}+\nabla\cdot R_{\mathrm{osc}}+\nabla\cdot R_{\mathrm{cor}}
\end{split}
\end{equation}
where $R_{\mathrm{lin}}$, $R_{\mathrm{osc}}$ and $R_{\mathrm{cor}}$ denote the linear error, oscillation error and corrector error respectively. 
We further analyze the oscillation term and reveal the crucial cancellations as follows
\begin{equation}\notag
\begin{split}
&\nabla\cdot(w_{q+1}^p\otimes w_{q+1}^p+R_q)+\partial_t v_{q+1}^c\\
=&\nabla\cdot\sum_{k\in\Lambda} a_{(k)}^2\fint_{\mathbb T^3} W_{(k)}\otimes W_{(k)}\, dx+\nabla\cdot R_q\\
&+\nabla\cdot\sum_{k\in\Lambda} a_{(k)}^2(g_{(\tau)}^2-1)\fint_{\mathbb T^3} W_{(k)}\otimes W_{(k)}\, dx+\partial_t v_{q+1}^c\\
&+\nabla\cdot\sum_{k\in\Lambda} a_{(k)}^2g_{(\tau)}^2\mathbb P_{\neq 0}\left(W_{(k)}\otimes W_{(k)}\right)\\
=&:\mathcal O_1+\mathcal O_2+\mathcal O_3.
\end{split}
\end{equation}
By virtue of the definition of $a_{(k)}$ in (\ref{def-a}), applying the Geometric Lemma \ref{le-geo} and the normalization property of $\phi$ to the first line yields
\begin{equation}\notag
\begin{split}
\mathcal O_1=&\ \nabla\cdot \sum_{k\in\Lambda}\rho\gamma_{(k)}^2(\mathrm{Id}-\frac{R_q}{\rho})k\otimes k+\nabla\cdot R_q\\
=&\ \nabla\cdot (\rho\mathrm{Id} -R_q)+\nabla\cdot R_q\\
=&\ \nabla \rho.
\end{split}
\end{equation}
Exploiting the definition of $v_{q+1}^c$ in (\ref{def-v-c}) we have
\begin{equation}\notag
\begin{split}
\mathcal O_2=&\ \sum_{k\in\Lambda} (g_{(\tau)}^2-1)\fint_{\mathbb T^3} W_{(k)}\otimes W_{(k)}\, dx \nabla a_{(k)}^2\\
&-\sigma^{-1}\sum_{k\in \Lambda}\mathbb P_{H}\mathbb P_{\neq 0}\left(\partial_th_{(\tau)}\fint_{\mathbb T^3}W_{(k)}\otimes W_{(k)}\,dx\nabla a_{(k)}^2 \right)\\
&-\sigma^{-1}\sum_{k\in \Lambda}\mathbb P_{H}\mathbb P_{\neq 0}\left(h_{(\tau)}\fint_{\mathbb T^3}W_{(k)}\otimes W_{(k)}\,dx\partial_t\nabla a_{(k)}^2 \right)\\
=&-\sigma^{-1}\sum_{k\in \Lambda}\mathbb P_{H}\mathbb P_{\neq 0}\left(h_{(\tau)}\fint_{\mathbb T^3}W_{(k)}\otimes W_{(k)}\,dx\partial_t\nabla a_{(k)}^2 \right)
\end{split}
\end{equation}
where we used (\ref{cancel-g}). Applying the fact $\nabla\cdot (W_{(k)}\otimes W_{(k)})=0$ gives us
\begin{equation}\notag
\mathcal O_3=\sum_{k\in\Lambda} \nabla a_{(k)}^2g_{(\tau)}^2\mathbb P_{\neq 0}\left(W_{(k)}\otimes W_{(k)}\right).
\end{equation}
Summarizing the analysis above we obtain
\begin{equation}\notag
\begin{split}
\nabla\cdot R_{\mathrm{osc}}=&-\sigma^{-1}\sum_{k\in \Lambda}\mathbb P_{H}\mathbb P_{\neq 0}\left(h_{(\tau)}\fint_{\mathbb T^3}W_{(k)}\otimes W_{(k)}\,dx\partial_t\nabla a_{(k)}^2 \right)\\
&+\sum_{k\in\Lambda} \nabla a_{(k)}^2g_{(\tau)}^2\mathbb P_{\neq 0}\left(W_{(k)}\otimes W_{(k)}\right)
\end{split}
\end{equation}
where we shift $\nabla\rho$ into the pressure term.

Let $\mathcal R$ be the inverse divergence operator. We can choose 
\begin{equation}\label{stress-concrete}
\begin{split}
R_{\mathrm{lin}}=&\ \mathcal R\partial_tv_{q+1}^p-\mathcal R\Delta v_{q+1}+\mathcal R\nabla\cdot(B_q\otimes w_{q+1}+w_{q+1}\otimes B_q),\\
R_{\mathrm{osc}}=& \sum_{k\in\Lambda} \mathcal R\mathbb P_{H}\mathbb P_{\neq 0}\left(g_{(\tau)}^2\mathbb P_{\neq 0}((W_{(k)}\otimes W_{(k)})\nabla a_{(k)}^2) \right)\\
&-\sigma^{-1}\sum_{k\in\Lambda}\mathcal R\mathbb P_{H}\mathbb P_{\neq 0}\left(h_{(\tau)}\fint_{\mathbb T^3}W_{(k)}\otimes W_{(k)}\, dx\partial_t\nabla a_{(k)}^2 \right)\\
=&\ R_{\mathrm{osc,1}}+R_{\mathrm{osc,2}},\\
R_{\mathrm{cor}}=&\ \mathcal R\mathbb P_H\nabla\cdot(w_{q+1}^c\otimes w_{q+1}+w_{q+1}^p\otimes w_{q+1}^c)
\end{split}
\end{equation}
and 
\begin{equation}\label{stress-final}
R_{q+1}=R_{\mathrm{lin}}+R_{\mathrm{osc,1}}+R_{\mathrm{osc,2}}+R_{\mathrm{cor}}.
\end{equation}

\subsection{Estimates of the new stress error}
\label{sec-error-est}
We start with the estimates for the linear errors. Note $\mathcal R\nabla\times$ is a Calder\'on-Zygmund operator. In view of the definition \eqref{def-per-v} of $v_{q+1}^p$, applying Lemma \ref{le-CZ}, H\"older's inequality, and estimates from Lemma \ref{le-est-phi}, Lemma \ref{le-est-g} and Lemma \ref{le-est-a}, we obtain
\begin{equation}\notag
\begin{split}
\|\mathcal \partial_tv_{q+1}^p\|_{L^1_tL^\eta_x}\lesssim &\sum_{k\in\Lambda}\|\mathcal R\nabla\times (\partial_ta_{(k)}g_{(\tau)}W_{(k)}^c)\|_{L^1_tL^\eta_x}\\
&+\sum_{k\in\Lambda}\|\mathcal R\nabla\times (a_{(k)}\partial_tg_{(\tau)}W_{(k)}^c)\|_{L^1_tL^\eta_x}\\
\lesssim&\sum_{k\in\Lambda}\|a_{(k)}\|_{C^1_{t,x}}\|g_{(\tau)}\|_{L^1_t}\|W_{(k)}^c\|_{C_tL_x^\eta}\\
&+\sum_{k\in\Lambda}\|a_{(k)}\|_{C_{t,x}}\|\partial_tg_{(\tau)}\|_{L^1_t}\|W_{(k)}^c\|_{C_tL_x^\eta}\\
\lesssim&\ \tau^{-\frac12}r^{\frac2\eta-1}\lambda_{q+1}^{-2}+\sigma\tau^{\frac12}r^{\frac2\eta-1}\lambda_{q+1}^{-2}.
\end{split}
\end{equation}
Similarly, by Lemma \ref{le-CZ}, Lemma \ref{le-est-wv-1} and the inductive assumption we have
\begin{equation}\notag
\|\mathcal R(-\Delta)^\alpha v_{q+1}^p\|_{L^1_tL^\eta_x}\lesssim \||\nabla|^{2\alpha-1} v_{q+1}^p\|_{L^1_tL^\eta_x}\lesssim \tau^{-\frac12}r^{\frac2\eta-1}\lambda_{q+1}^{2\alpha-2},
\end{equation}
\begin{equation}\notag
\|\mathcal R(-\Delta)^\alpha v_{q+1}^c\|_{L^1_tL^\eta_x}\lesssim \||\nabla|^{2\alpha-1} v_{q+1}^c\|_{L^1_tL^\eta_x}\lesssim \sigma^{-1},
\end{equation}
and
\begin{equation}\notag
\begin{split}
\|\mathcal R\div (B_q\otimes w_{q+1}+w_{q+1}\otimes B_q)\|_{L^1_tL^\eta_x}\lesssim& \|B_q\otimes w_{q+1}+w_{q+1}\otimes B_q\|_{L^1_tL^\eta_x}\\
\lesssim&\|B_q\|_{L^\infty_{t,x}}\|w_{q+1}\|_{L^1_tL^\eta_x}\\
\lesssim&\ \lambda_q^5 \tau^{-\frac12}r^{\frac2\eta-1}.
\end{split}
\end{equation}

Applying H\"older's inequality, Lemma \ref{le-est-g} and Lemma \ref{le-est-a} to the oscillation errors in \eqref{stress-concrete}, we obtain
\begin{equation}\notag
\begin{split}
\|R_{\mathrm{osc,1}}\|_{L^1_tL^\eta_x}\lesssim&\sum_{k\in\Lambda}\|g_{(\tau)}^2\|_{L^1_t}\||\nabla|^{-1}\mathbb P_{\neq 0}\mathbb P_{\geq \lambda_{q+1}r/2}\left((W_{(k)}\otimes W_{(k)})\nabla a_{(k)}^2 \right)\|_{C_tL^\eta_x}\\
\lesssim&\sum_{k\in\Lambda}\|g_{(\tau)}\|_{L^2_t}^2\lambda_{q+1}^{-1}r^{-1}\|\nabla a_{(k)}\|_{C_{t,x}}\|\phi_{(k)}\|_{L^1_tL^{2\eta}_x}^2\\
\lesssim&\ \lambda_{q+1}^{-1}r^{\frac2\eta-3},
\end{split}
\end{equation}
and
\begin{equation}\notag
\begin{split}
\|R_{\mathrm{osc,2}}\|_{L^1_tL^\eta_x}\lesssim&\ \sigma^{-1}\sum_{k\in\Lambda}\|h_{(\tau)}\|_{C_t}\left( \|a_{(k)}\|_{C_{t,x}}\|\nabla a_{(k)}\|_{C^1_{t,x}}+\| a_{(k)}\|_{C^1_{t,x}}^2\right)\\
\lesssim&\ \sigma^{-1}.
\end{split}
\end{equation}

We estimate the corrector error using the estimates from Lemma \ref{le-est-wv-1}
\begin{equation}\notag
\begin{split}
\|R_{\mathrm{cor}}\|_{L^1_tL^\eta_x}\lesssim & \|w_{q+1}^c\otimes w_{q+1}+w_{q+1}^p\otimes w_{q+1}^c\|_{L^1_tL^\eta_x}\\
\lesssim& \|w_{q+1}^c\|_{L^2_tL^\infty_x}\left(\|w_{q+1}^p\|_{L^2_tL^\eta_x} +\|w_{q+1}\|_{L^2_tL^\eta_x}\right)\\
\lesssim&\ \lambda_{q+1}^{-1}r^{-1}\left(r^{\frac2\eta-1}+\lambda_{q+1}^{-1}r^{\frac2\eta-1} \right)\\
\lesssim&\ \lambda_{q+1}^{-1}r^{\frac2\eta-2}. 
\end{split}
\end{equation}

\subsection{Choice of parameters}
\label{sec-para-1}
To ensure the iterative scheme moving forward, we need to guarantee \eqref{induct-R}, \eqref{iter-2} and \eqref{iter-3}, i.e.
\[\|R_{q+1}\|_{L^1_tL^\eta_x}\leq \delta_{q+2}, \ \ \|w_{q+1}\|_{L^1_t L^2_x}\leq \delta_{q+2}^{\frac12}, \ \ \|w_{q+1}\|_{L^\gamma_t W^{1,p}_x}\leq \delta_{q+2}^{\frac12}. \]
Thus, collecting the estimates in Subsection \ref{sec-error-est} to be applied to \eqref{stress-final}, we impose
\begin{equation}\label{para-cond1}
\begin{split}
C\left(\sigma\tau^{\frac12}r^{\frac2\eta-1}\lambda_{q+1}^{-2}+\tau^{-\frac12}r^{\frac2\eta-1}\lambda_{q+1}^{2\alpha-2} +\lambda_q^5\tau^{-\frac12}r^{\frac2\eta-1}\right.\\
\left.+\lambda_{q+1}^{-1}r^{\frac2\eta-3}+\lambda_{q+1}^{-1}r^{-1}+\sigma^{-1} \right)\leq&\ \delta_{q+2},\\
C\tau^{-\frac12}\leq &\ \delta_{q+2}^{\frac12},\\
Cr^{\frac2p-1}\tau^{\frac12-\frac1\gamma}\lambda_{q+1}\leq &\ \delta_{q+2}^{\frac12}
\end{split}
\end{equation}
for some constant $C>0$. Recall
\[r=\lambda_{q+1}^{n_1}, \ \ \tau=\lambda_{q+1}^{n_2}, \ \ \sigma=\lambda_{q+1}^{2\ve}, \ \ \delta_{q+2}=\lambda_{q+1}^{-2b\beta}.\]
The conditions in \eqref{para-cond1} will be satisfied provided 
\begin{equation}\notag
\begin{split}
(\frac2\eta-1)n_1+\frac12n_2+2\ve-2<&-2b\beta,\\
(\frac2\eta-1)n_1-\frac12n_2+2\alpha-2<&-2b\beta,\\
5/b+(\frac2\eta-1)n_1-\frac12n_2<&-2b\beta,\\
-n_1-1<&-2b\beta,\\
(\frac2\eta-3)n_1-1<&-2b\beta,\\
 -2\ve <&-2b\beta,\\
 -\frac12n_2 <&-b\beta, \\
(\frac2p-1)n_1+(\frac12-\frac1\gamma)n_2+1<&-b\beta.
\end{split}
\end{equation}
Since $\eta$ can be chosen as close as to $1$, we take $\eta=1$ in the conditions for brevity and hence obtain
\begin{equation}\label{para-1}
\begin{split}
n_1+\frac12n_2+2\ve-2<&-2b\beta,\\
n_1-\frac12n_2+2\alpha-2<&-2b\beta,\\
5/b+n_1-\frac12n_2<&-2b\beta,\\
-n_1-1<&-2b\beta,\\
-2\ve <&-2b\beta,\\
 -\frac12n_2 <&-b\beta, \\
(\frac2p-1)n_1+(\frac12-\frac1\gamma)n_2+1<&-b\beta.
\end{split}
\end{equation}
Note that the fifth condition of (\ref{para-1}) implies $\ve>b\beta$. As discussed in the heuristic analysis in Subsection \ref{sec-heuristics}, we do not need much spatial intermittency (corresponding to spatial concentration). Hence we choose $n_1=-3\ve$. We also take large enough $b>0$.
For the first two conditions of \eqref{para-1} to be compatible, we impose $\alpha<2+2\ve-2b\beta$. In the end, we take $n_2=(4+2\varepsilon-4b\beta)-\epsilon_0$ with arbitrarily small $\epsilon_0>0$ such that the first six conditions of (\ref{para-1}) are all valid.

Taking $p=\infty$ in the last inequality of (\ref{para-1}) yields
\[\frac1\gamma>\frac12+\frac{1+3\ve+b\beta}{n_2}=\frac12+\frac{1+3\ve+b\beta}{(4+2\varepsilon-4b\beta)-\epsilon_0}>\frac34\]
for sufficiently small $\ve>b\beta>0$. Recall the critical space with $p=\infty$ corresponds to $\frac1\gamma= \frac{2\alpha-1}{2\alpha}$. We observe that 
\[\frac{2\alpha-1}{2\alpha}=\frac34 \ \ \mbox{for} \ \ \alpha=2.\]
Hence we note that the scheme gives non-unique weak solutions in the boarder line space $L_t^\gamma W^{1,\infty}_x$ for $\gamma<\frac43$ for the hyper-resistive equation (\ref{emhd}) with $\alpha=2$.

To summarize, with the new stress error $R_{q+1}$ defined through \eqref{stress-concrete}-\eqref{stress-final} and appealing to \eqref{R-q1-initial}, it is apparent that the triplet $(A_{q+1}, B_{q+1}, R_{q+1})$ is a solution of \eqref{emhd-q} at the $(q+1)$-th level. The estimate \eqref{iter-1} follows from \eqref{est-w-2} and the inductive assumption \eqref{induct-R}. In Subsection \ref{sec-para-1}, the analysis shows that the estimates \eqref{iter-2} and \eqref{iter-3} are satisfied under appropriate choice of parameters; it also proves \eqref{induct-R} with $q$ replaced by $(q+1)$. The estimate \eqref{induct-B} (and \eqref{induct-A}) with $q$ replaced by $(q+1)$ is obvious by noticing that
\begin{equation}\notag
\|B_{q+1}\|_{L^\infty_tH^2_x}\leq \|B_{q}\|_{L^\infty_tH^2_x}+\|w_{q+1}\|_{L^\infty_tH^2_x}\lesssim \lambda_q^5+\lambda_{q+1}^2\tau^{\frac12}\lesssim \lambda_{q+1}^5.
\end{equation}
Regarding \eqref{iter-supp}, we have
\begin{equation}\notag
\begin{split}
\supp_t (A_{q+1},B_{q+1}, \mathring R_{q+1})\subset&\ \supp_t (A_{q},B_{q}, \mathring R_{q})\cup \supp_t (w_{q+1})\\
\subset&\ \supp_t (A_{q},B_{q}, \mathring R_{q})\cup (\cup_{k\in\Lambda} \supp_t a_{(k)})\\
\subset&\ \supp_t (A_{q},B_{q}, \mathring R_{q})\cup N_{\delta_{q+2}^{\frac12}}(\supp_t \mathring R_{q})\\
\subset&\ N_{\delta_{q+2}^{\frac12}}(\supp_t (A_{q},B_{q}, \mathring R_{q})).
\end{split}
\end{equation}
The proof of Proposition \ref{prop} in the case of $L^\gamma_tW^{1,\infty}_x$ for $\gamma<\frac43$ and $\alpha\in[1,2]$ is complete.


\begin{rem}\label{rem-case1}
If we consider the electron MHD \eqref{emhd} without resistivity $(-\Delta)^\alpha B$, the condition 
\[n_1-\frac12n_2+2\alpha-2<-2b\beta\]
from \eqref{para-1} is not needed. Without considering this condition, the rest analysis remains the same except that we do not impose $\alpha<2+2\ve-2b\beta$ anymore. 
\end{rem}

\bigskip

\section{Proof of the iteration proposition of case II: boarder line space $L_t^\infty W_x^{1,p}$}
\label{sec-case-2}

\subsection{Building blocks}
\label{sec-blocks-2}
In this case, we need to construct solutions in $L_t^\infty W_x^{1,p}$ where spatial integrability is weaker, the analysis in Subsection \ref{sec-heuristics} suggests adapting intermittent jets as our building blocks. Beside the concentration in the orthogonal plane of a direction $k\in \Lambda$, we also need to introduce concentration and temporal oscillation in the parallel direction of $k$. Let $\ell$ and $\mu$ denote the concentration in the direction of $k$ and temporal oscillation respectively. We choose a smooth and mean-zero function $\psi:\mathbb R\to \mathbb R$ satisfying
\begin{equation}\label{norm-psi}
\frac1{2\pi}\int_{\mathbb R}\psi^2(x)\, dx=1, \ \ \ \supp \psi\subset [-1,1].
\end{equation}
As before, we periodize the rescaled function
\[\psi_{\ell}(x)=\ell^{-\frac12}\psi(\frac{x}{\ell})\]
and treat it as a periodic function on $\mathbb T$. Denote 
\[\psi_{(k)}(x,t)=\psi_{\ell}(\lambda_{q+1} rN_\Lambda (k\cdot x+\mu t)).\]
We then define the intermittent jets as 
\begin{equation}\label{W-2}
W_{(k)}(x,t)=\psi_{(k)}\phi_{(k)}k, \ \ k\in \Lambda
\end{equation}
with $\phi_{(k)}$ from Section \ref{sec-case-1}. We also define 
\[W_{(k)}^c=\frac{1}{\lambda_{q+1}^2N_{\Lambda}^2}\psi_{(k)}\Phi_{(k)}k.\]
We observe
\[\nabla\times \nabla\times W_{(k)}^c=W_{(k)}+\widetilde W_{(k)}\]
with 
\[\widetilde W_{(k)}=\frac{1}{\lambda_{q+1}^2N_{\Lambda}^2}\nabla\psi_{(k)}\times (\nabla\times (\Phi_{(k)}k)).\]
It is obvious that 
\[\div (W_{(k)}+\widetilde W_{(k)})=0.\]

\begin{lem}\label{le-est-psi}
For $N, M\in\mathbb N$ and $p\in[1,\infty]$ we have
\begin{equation}\notag
\begin{split}
\|\nabla^N\partial_t^M\psi_{(k)}\|_{C_tL^p_x}\lesssim \ell^{\frac1p-\frac12}\left(\frac{\lambda_{q+1}r}{\ell}\right)^N\left(\frac{\lambda_{q+1}r\mu}{\ell}\right)^M,\\
\|\nabla^N\partial_t^MW_{(k)}\|_{C_tL^p_x}+\frac{\ell}{r}\|\nabla^N\partial_t^M\widetilde W_{(k)}\|_{C_tL^p_x}+\lambda_{q+1}^2\|\nabla^N\partial_t^MW^c_{(k)}\|_{C_tL^p_x}\lesssim  r^{\frac2p-1}\ell^{\frac1p-\frac12}\lambda_{q+1}^N\left(\frac{\lambda_{q+1}r\mu}{\ell}\right)^M
\end{split}
\end{equation}
with implicit constants dependent of $N_\Lambda$ and independent of $r, \ell, \mu$ and $\lambda_{q+1}$.
\end{lem}
The proof of the lemma is standard and thus omitted. 

We define the temporal building blocks $g_{(\tau)}$, $h_{(\tau)}$ and amplitude functions $a_{(k)}$ as in Section \ref{sec-case-1}. 

\subsection{Perturbations of magnetic and potential fields}
As before we first define the principal perturbation for the vector potential
\begin{equation}\label{def-per-v}
v_{q+1}^{p}=\sum_{k\in\Lambda}\nabla\times (a_{(k)}g_{(\tau)}W_{(k)}^c)
\end{equation}
and the associated perturbation for the magnetic field
\begin{equation}\label{def-per-w}
\begin{split}
\nabla\times v_{q+1}^{p}=&\sum_{k\in\Lambda}\nabla\times\nabla\times (a_{(k)}g_{(\tau)}W_{(k)}^c)\\
=& \sum_{k\in\Lambda}a_{(k)}g_{(\tau)}W_{(k)}\\
&+\sum_{k\in \Lambda}g_{(\tau)}\left(\nabla a_{(k)}\times (\nabla\times W_{(k)}^c)+\nabla\times (\nabla a_{(k)}\times W_{(k)}^c)+a_{(k)}\widetilde W_{(k)}\right)\\
=&: w_{q+1}^{p}+w_{q+1}^{c}.
\end{split}
\end{equation}
Again it is obvious that
\[\nabla\cdot v_{q+1}^{p}=0, \ \ \nabla\cdot (w_{q+1}^{p}+w_{q+1}^{c})=0.\]
Beside the temporal corrector for the magnetic vector potential 
\begin{equation}\notag
v_{q+1}^{c}=-\sigma^{-1}\sum_{k\in \Lambda}\mathbb P_{H}\mathbb P_{\neq 0}\left(h_{(\tau)}\fint_{\mathbb T^3}W_{(k)}\otimes W_{(k)}\,dx\nabla a_{(k)}^2 \right)
\end{equation}
introduced in Section \ref{sec-case-1} to cancel a high temporal oscillation in the interactions, we need to add one more temporal corrector to reduce the term containing $\div(W_{(k)}\otimes W_{(k)})$ which does not vanish for the $W_{(k)}$ defined in (\ref{W-2}). In particular, we define the second temporal corrector as
\begin{equation}\label{v-t}
v_{q+1}^t=-\mu^{-1}\sum_{k\in \Lambda}\mathbb P_H\mathbb P_{\neq 0}\left( a_{(k)}^2g_{(\tau)}^2\psi_{(k)}^2\phi_{(k)}^2k\right).
\end{equation}
It follows that
\begin{equation}\label{v-t-eq}
\begin{split}
&\partial_t v_{q+1}^{t}+\sum_{k\in \Lambda}\mathbb P_{\neq 0}\left( a_{(k)}^2g_{(\tau)}^2\div(W_{(k)}\otimes W_{(k)})\right)\\
=&\ \mu^{-1}\nabla \Delta^{-1}\div \sum_{k\in \Lambda}\mathbb P_{\neq 0}\partial_t\left( a_{(k)}^2g_{(\tau)}^2\psi_{(k)}^2\phi_{(k)}^2k\right)\\
&-\mu^{-1}\sum_{k\in \Lambda}\mathbb P_{\neq 0}\left( \partial_t(a_{(k)}^2g_{(\tau)}^2)\psi_{(k)}^2\phi_{(k)}^2k\right).
\end{split}
\end{equation}
Naturally we define one more corrector 
\[w_{q+1}^t=\nabla\times v_{q+1}^t\]
for the magnetic field.

In summary, the total perturbations of the magnetic field and its vector potential are defined as
\begin{equation}\label{def-per-wv}
w_{q+1}=w_{q+1}^{p}+w_{q+1}^{c}+w_{q+1}^{t}, \ \ v_{q+1}=v_{q+1}^{p}+v_{q+1}^{c}+v_{q+1}^{t}.
\end{equation}
It is clear to see $\nabla\times v_{q+1}^c=0$ and hence $w_{q+1}=\nabla\times v_{q+1}$. We then define the new magnetic field and its vector potential as
\[B_{q+1}=B_q+w_{q+1}, \ \ A_{q+1}=A_q+v_{q+1}.\]

\begin{lem}\label{le-est-wv}
For $\gamma\in[1,\infty]$, $\eta\in(1,\infty)$ and $0\leq N\leq 8$, the estimates
\begin{equation}\label{est-w-p}
\|\nabla^N w_{q+1}^p\|_{L_t^\gamma L_x^\eta}\lesssim \lambda_{q+1}^Nr^{\frac2\eta-1}\ell^{\frac1\eta-\frac12}\tau^{\frac12-\frac1\gamma},
\end{equation}
\begin{equation}\label{est-w-c}
\|\nabla^N w_{q+1}^c\|_{L_t^\gamma L_x^\eta}\lesssim \lambda_{q+1}^{N}r^{\frac2\eta}\ell^{\frac1\eta-\frac32}\tau^{\frac12-\frac1\gamma},
\end{equation}
\begin{equation}\label{est-w-t}
\|\nabla^N w_{q+1}^t\|_{L_t^\gamma L_x^\eta}\lesssim \mu^{-1}\lambda_{q+1}^{N+1}r^{\frac2\eta-2}\ell^{\frac1\eta-1}\tau^{1-\frac1\gamma},
\end{equation}
\begin{equation}\label{est-v-p}
\|\nabla^N v_{q+1}^p\|_{L_t^\gamma L_x^\eta}\lesssim \lambda_{q+1}^{N-1}r^{\frac2\eta-1}\ell^{\frac1\eta-\frac12}\tau^{\frac12-\frac1\gamma},
\end{equation}
\begin{equation}\label{est-v-c}
\|\nabla^N v_{q+1}^c\|_{L_t^\gamma L_x^\eta}\lesssim \sigma^{-1},
\end{equation}
\begin{equation}\label{est-v-p}
\|\nabla^N v_{q+1}^t\|_{L_t^\gamma L_x^\eta}\lesssim \mu^{-1}\lambda_{q+1}^{N}r^{\frac2\eta-2}\ell^{\frac1\eta-1}\tau^{1-\frac1\gamma}
\end{equation}
hold with implicit constants depending only on $N$, $\gamma$ and $\eta$. 
\end{lem}
\begin{proof}
By H\"older's inequality, Lemma \ref{le-est-g}, Lemma \ref{le-est-a} and Lemma \ref{le-est-psi} we have
\begin{equation}\notag
\begin{split}
\|\nabla^Nw_{q+1}^p\|_{L^\gamma_t L^\eta_x}\lesssim&\sum_{k\in \Lambda}\sum_{N_1+N_2=N}\|a_{(k)}\|_{C^{N_1}_{t,x}}\|g_{(\tau)}\|_{L^\gamma_t}\|\nabla^{N_2}W_{(k)}\|_{L^\infty_tL^\eta_x}\\
\lesssim&\sum_{N_1+N_2=N}\tau^{\frac12-\frac1\gamma}\lambda_{q+1}^{N_2}r^{\frac2\eta-1}\ell^{\frac1\eta-\frac12}\\
\lesssim&\ \lambda_{q+1}^Nr^{\frac2\eta-1}\ell^{\frac1\eta-\frac12}\tau^{\frac12-\frac1\gamma}.
\end{split}
\end{equation}
Similarly, it follows from \eqref{def-per-v}, Lemma \ref{le-est-g}, Lemma \ref{le-est-a} and Lemma \ref{le-est-psi} that
\begin{equation}\notag
\begin{split}
\|\nabla^Nv_{q+1}^p\|_{L^\gamma_t L^\eta_x}\lesssim&\sum_{k\in \Lambda}\sum_{N_1+N_2=N+1}\|a_{(k)}\|_{C^{N_1}_{t,x}}\|g_{(\tau)}\|_{L^\gamma_t}\|\nabla^{N_2}W_{(k)}^c\|_{L^\infty_tL^\eta_x}\\
\lesssim&\sum_{N_1+N_2=N+1}\tau^{\frac12-\frac1\gamma}\lambda_{q+1}^{N_2-2}r^{\frac2\eta-1}\ell^{\frac1\eta-\frac12}\\
\lesssim&\ \lambda_{q+1}^{N-1}r^{\frac2\eta-1}\ell^{\frac1\eta-\frac12}\tau^{\frac12-\frac1\gamma}.
\end{split}
\end{equation}
Applying Lemma \ref{le-est-g}, Lemma \ref{le-est-a} and Lemma \ref{le-est-psi} to $w_{q+1}^c$ as defined in \eqref{def-per-w} yields
\begin{equation}\notag
\begin{split}
\|\nabla^Nw_{q+1}^c\|_{L^\gamma_t L^\eta_x}\lesssim&\sum_{k\in \Lambda}\sum_{N_1+N_2=N}\|a_{(k)}\|_{C^{N_1}_{t,x}}\|g_{(\tau)}\|_{L^\gamma_t}\\
&\cdot\left(\|\nabla^{N_2}W_{(k)}^c\|_{L^\infty_tL^\eta_x} +\|\nabla^{N_2}\nabla W_{(k)}^c\|_{L^\infty_tL^\eta_x}+\|\nabla^{N_2}\widetilde W_{(k)}\|_{L^\infty_tL^\eta_x}\right)\\
\lesssim&\sum_{N_1+N_2=N}\tau^{\frac12-\frac1\gamma}\left(\lambda_{q+1}^{N_2-2}r^{\frac2\eta-1}\ell^{\frac1\eta-\frac12}+\lambda_{q+1}^{N_2-1}r^{\frac2\eta-1}\ell^{\frac1\eta-\frac12}+r\ell^{-1}r^{\frac2\eta-1}\ell^{\frac1\eta-\frac12}\lambda_{q+1}^{N_2} \right)\\
\lesssim&\ \lambda_{q+1}^{N}r^{\frac2\eta}\ell^{\frac1\eta-\frac32}\tau^{\frac12-\frac1\gamma}.
\end{split}
\end{equation}
For $v_{q+1}^t$ defined in \eqref{v-t}, we have from Lemma \ref{le-est-g}, Lemma \ref{le-est-a}, Lemma \ref{le-est-phi} and Lemma \ref{le-est-psi}
\begin{equation}\notag
\begin{split}
\|\nabla^N v_{q+1}^t\|_{L_t^\gamma L_x^\eta}\lesssim &\ \mu^{-1}\sum_{k\in\Lambda}\|g_{(\tau)}^2\|_{L^\gamma_t}\sum_{N_1+N_2+N_3=N}\|\nabla^{N_1}a_{(k)}^2\|_{C_{t,x}}\|\nabla^{N_2}\psi_{(k)}^2\|_{C_tL_x^\eta}\|\nabla^{N_3}\phi_{(k)}^2\|_{C_tL_x^\eta}\\
\lesssim &\ \mu^{-1}\sum_{k\in\Lambda}\|g_{(\tau)}\|_{L^{2\gamma}_t}^2\sum_{N_1+N_2+N_3=N}\lambda_{q+1}^{N_2+N_3}\|\nabla^{N_1}a_{(k)}^2\|_{C_{t,x}}\|\psi_{(k)}\|_{C_tL_x^{2\eta}}^2\|\phi_{(k)}\|_{C_tL_x^{2\eta}}^2\\
\lesssim&\ \mu^{-1}\tau^{1-\frac1\gamma}\lambda_{q+1}^Nr^{\frac2\eta-2}\ell^{\frac1\eta-1},
\end{split}
\end{equation}
and
\begin{equation}\notag
\|\nabla^N w_{q+1}^t\|_{L_t^\gamma L_x^\eta}\lesssim  \lambda_{q+1}\|\nabla^N v_{q+1}^t\|_{L_t^\gamma L_x^\eta}\lesssim \mu^{-1}\tau^{1-\frac1\gamma}\lambda_{q+1}^{N+1}r^{\frac2\eta-2}\ell^{\frac1\eta-1}.
\end{equation}

The estimate \eqref{est-v-c} follows immediately from the definition of $v_{q+1}^c$ and the estimates from Lemma \ref{le-est-g} and Lemma \ref{le-est-a}.


\end{proof}


\medskip

\subsection{New stress tensor}
The new stress error $R_{q+1}$ such that $(A_{q+1}, B_{q+1}, R_{q+1})$ solves system (\ref{emhd-q}) at the $(q+1)$-th level satisfies
\begin{equation}\notag
\begin{split}
\nabla\cdot R_{q+1}-\nabla P_{q+1}=&\ \partial_tv_{q+1}^p+(-\Delta)^\alpha v_{q+1}+\nabla\cdot(B_q\otimes w_{q+1}+w_{q+1}\otimes B_q)\\
&+\nabla\cdot(w_{q+1}^p\otimes w_{q+1}^p+R_q)+\partial_t v_{q+1}^c+\partial_t v_{q+1}^t\\
&+\nabla\cdot((w_{q+1}^c+w_{q+1}^t)\otimes w_{q+1}+w_{q+1}^p\otimes (w_{q+1}^c+w_{q+1}^t))\\
=&:\nabla\cdot R_{\mathrm{lin}}+\nabla\cdot R_{\mathrm{osc}}+\nabla\cdot R_{\mathrm{cor}}.
\end{split}
\end{equation}
Exploiting the cancellations in the oscillation term gives
\begin{equation}\notag
\begin{split}
&\nabla\cdot(w_{q+1}^p\otimes w_{q+1}^p+R_q)+\partial_t v_{q+1}^c+\partial_t v_{q+1}^t\\
=&\nabla\cdot\sum_{k\in\Lambda} a_{(k)}^2\fint_{\mathbb T^3} W_{(k)}\otimes W_{(k)}\, dx+\nabla\cdot R_q\\
&+\nabla\cdot\sum_{k\in\Lambda} a_{(k)}^2(g_{(\tau)}^2-1)\fint_{\mathbb T^3} W_{(k)}\otimes W_{(k)}\, dx+\partial_t v_{q+1}^c\\
&+\nabla\cdot\sum_{k\in\Lambda} a_{(k)}^2g_{(\tau)}^2\mathbb P_{\neq 0}\left(W_{(k)}\otimes W_{(k)}\right)+\partial_t v_{q+1}^t\\
=&:\mathcal O_1+\mathcal O_2+\mathcal O_3.
\end{split}
\end{equation}
Similar analysis as in Section \ref{sec-case-1} gives 
\begin{equation}\notag
\begin{split}
\mathcal O_1=&\ \nabla \rho,\\
\mathcal O_2=&-\sigma^{-1}\sum_{k\in \Lambda}\mathbb P_{H}\mathbb P_{\neq 0}\left(h_{(\tau)}\fint_{\mathbb T^3}W_{(k)}\otimes W_{(k)}\,dx\partial_t\nabla a_{(k)}^2 \right).
\end{split}
\end{equation}
In view of the definition of $v_{q+1}^t$ in (\ref{v-t}) and (\ref{v-t-eq}) we obtain
\begin{equation}\notag
\begin{split}
\mathcal O_3=&\sum_{k\in\Lambda} \nabla a_{(k)}^2g_{(\tau)}^2\mathbb P_{\neq 0}\left(W_{(k)}\otimes W_{(k)}\right)
+\sum_{k\in\Lambda}  a_{(k)}^2g_{(\tau)}^2\mathbb P_{\neq 0}\div\left(W_{(k)}\otimes W_{(k)}\right)+\partial_tv_{q+1}^t\\
=&\sum_{k\in\Lambda} \nabla a_{(k)}^2g_{(\tau)}^2\mathbb P_{\neq 0}\left(W_{(k)}\otimes W_{(k)}\right)
+ \mu^{-1}\nabla \Delta^{-1}\div \sum_{k\in \Lambda}\mathbb P_{\neq 0}\partial_t\left( a_{(k)}^2g_{(\tau)}^2\psi_{(k)}^2\phi_{(k)}^2k\right)\\
&-\mu^{-1}\sum_{k\in \Lambda}\mathbb P_{\neq 0}\left( \partial_t(a_{(k)}^2g_{(\tau)}^2)\psi_{(k)}^2\phi_{(k)}^2k\right).
\end{split}
\end{equation}
Note that $\mathcal O_1$ and the second term in $\mathcal O_3$ are in gradient form and can be put in the pressure term $\nabla P_{q+1}$. Therefore we have
\begin{equation}\notag
\begin{split}
\nabla\cdot R_{\mathrm{osc}}=&-\sigma^{-1}\sum_{k\in \Lambda}\mathbb P_{H}\mathbb P_{\neq 0}\left(h_{(\tau)}\fint_{\mathbb T^3}W_{(k)}\otimes W_{(k)}\,dx\partial_t\nabla a_{(k)}^2 \right)\\
&+\sum_{k\in\Lambda} \nabla a_{(k)}^2g_{(\tau)}^2\mathbb P_{\neq 0}\left(W_{(k)}\otimes W_{(k)}\right) -\mu^{-1}\sum_{k\in \Lambda}\mathbb P_{\neq 0}\left( \partial_t(a_{(k)}^2g_{(\tau)}^2)\psi_{(k)}^2\phi_{(k)}^2k\right).
\end{split}
\end{equation}
It follows that we can choose 
\begin{equation}\notag
\begin{split}
R_{\mathrm{lin}}=&\ \mathcal R\partial_tv_{q+1}^p+\mathcal R(-\Delta)^\alpha v_{q+1}+\mathcal R\nabla\cdot(B_q\otimes w_{q+1}+w_{q+1}\otimes B_q),\\
R_{\mathrm{osc}}=& \sum_{k\in\Lambda} \mathcal R\mathbb P_{H}\mathbb P_{\neq 0}\left(g_{(\tau)}^2\mathbb P_{\neq 0}((W_{(k)}\otimes W_{(k)})\nabla a_{(k)}^2) \right)\\
&-\sigma^{-1}\sum_{k\in\Lambda}\mathcal R\mathbb P_{H}\mathbb P_{\neq 0}\left(h_{(\tau)}\fint_{\mathbb T^3}W_{(k)}\otimes W_{(k)}\, dx\partial_t\nabla a_{(k)}^2 \right)\\
&-\mu^{-1}\sum_{k\in \Lambda}\mathcal R\mathbb P_{H}\mathbb P_{\neq 0}\left( \partial_t(a_{(k)}^2g_{(\tau)}^2)\psi_{(k)}^2\phi_{(k)}^2k\right)\\
=&\ R_{\mathrm{osc,1}}+R_{\mathrm{osc,2}}+R_{\mathrm{osc,3}},\\
R_{\mathrm{cor}}=&\ \mathcal R\mathbb P_H\nabla\cdot((w_{q+1}^c+w_{q+1}^t)\otimes w_{q+1}+w_{q+1}^p\otimes (w_{q+1}^c+w_{q+1}^t))
\end{split}
\end{equation}
and 
\[R_{q+1}=R_{\mathrm{lin}}+R_{\mathrm{osc,1}}+R_{\mathrm{osc,2}}+R_{\mathrm{osc,3}}+R_{\mathrm{cor}}.\]

\subsection{Estimates of the new stress error}
Appealing to \eqref{def-per-v}, the first linear error is estimated by using Lemma \ref{le-CZ}, Lemma \ref{le-est-g}, Lemma \ref{le-est-a} and Lemma \ref{le-est-psi} 
\begin{equation}\notag
\begin{split}
\|\mathcal \partial_tv_{q+1}^p\|_{L^1_tL^\eta_x}\lesssim &\sum_{k\in\Lambda}\|\mathcal R\nabla\times (\partial_ta_{(k)}g_{(\tau)}W_{(k)}^c)\|_{L^1_tL^\eta_x}+\sum_{k\in\Lambda}\|\mathcal R\nabla\times (a_{(k)}\partial_tg_{(\tau)}W_{(k)}^c)\|_{L^1_tL^\eta_x}\\
&+\sum_{k\in\Lambda}\|\mathcal R\nabla\times (a_{(k)}g_{(\tau)}\partial_tW_{(k)}^c)\|_{L^1_tL^\eta_x}\\
\lesssim&\sum_{k\in\Lambda}\|a_{(k)}\|_{C^1_{t,x}}\|g_{(\tau)}\|_{L^1_t}\|W_{(k)}^c\|_{C_tL_x^\eta}+\sum_{k\in\Lambda}\|a_{(k)}\|_{C_{t,x}}\|\partial_tg_{(\tau)}\|_{L^1_t}\|W_{(k)}^c\|_{C_tL_x^\eta}\\
&+\sum_{k\in\Lambda}\|a_{(k)}\|_{C_{t,x}}\|g_{(\tau)}\|_{L^1_t}\|\partial_tW_{(k)}^c\|_{C_tL_x^\eta}\\
\lesssim&\ \tau^{-\frac12}r^{\frac2\eta-1}\ell^{\frac1\eta-\frac12}\lambda_{q+1}^{-2}+\sigma\tau^{\frac12}r^{\frac2\eta-1}\ell^{\frac1\eta-\frac12}\lambda_{q+1}^{-2}+\tau^{-\frac12}r^{\frac2\eta-1}\ell^{\frac1\eta-\frac12}\lambda_{q+1}^{-2}\frac{\lambda_{q+1}r\mu}{\ell}\\
\lesssim&\ \tau^{-\frac12}r^{\frac2\eta-1}\ell^{\frac1\eta-\frac12}\lambda_{q+1}^{-2}\left(1+\sigma\tau+ \frac{\lambda_{q+1}r\mu}{\ell}\right).
\end{split}
\end{equation}
Similarly we have, by using the definitions of $v_{q+1}^p$, $v_{q+1}^c$ and $v_{q+1}^t$
\begin{equation}\notag
\|\mathcal R(-\Delta)^\alpha v_{q+1}^p\|_{L^1_tL^\eta_x}\lesssim \||\nabla|^{2\alpha-1} v_{q+1}^p\|_{L^1_tL^\eta_x}\lesssim \tau^{-\frac12}r^{\frac2\eta-1}\ell^{\frac1\eta-\frac12}\lambda_{q+1}^{2\alpha-2},
\end{equation}
\begin{equation}\notag
\|\mathcal R(-\Delta)^\alpha v_{q+1}^c\|_{L^1_tL^\eta_x}\lesssim \||\nabla|^{2\alpha-1} v_{q+1}^c\|_{L^1_tL^\eta_x}\lesssim \sigma^{-1},
\end{equation}
\begin{equation}\notag
\|\mathcal R(-\Delta)^\alpha v_{q+1}^t\|_{L^1_tL^\eta_x}\lesssim \||\nabla|^{2\alpha-1} v_{q+1}^t\|_{L^1_tL^\eta_x}\lesssim \mu^{-1}r^{\frac2\eta-2}\ell^{\frac1\eta-1}\lambda_{q+1}^{2\alpha-1}.
\end{equation}
Applying Lemma \ref{le-CZ}, the inductive assumption \eqref{induct-B} and Lemma \ref{le-est-wv} we obtain
\begin{equation}\notag
\begin{split}
&\|\mathcal R\mathbb P_H\div (B_q\otimes w_{q+1}+w_{q+1}\otimes B_q)\|_{L^1_tL^\eta_x}\\
\lesssim& \|B_q\otimes w_{q+1}+w_{q+1}\otimes B_q\|_{L^1_tL^\eta_x}\\
\lesssim&\|B_q\|_{L^\infty_{t,x}}\|w_{q+1}\|_{L^1_tL^\eta_x}\\
\lesssim&\ \lambda_q^5 \left(\tau^{-\frac12}r^{\frac2\eta-1}\ell^{\frac1\eta-\frac12}+\tau^{-\frac12}r^{\frac{2}{\eta}}\ell^{\frac1\eta-\frac32}+\lambda_{q+1}\mu^{-1}r^{\frac2\eta-2}\ell^{\frac1\eta-1}\right).
\end{split}
\end{equation}
Applying Lemma \ref{le-est-g} and Lemma \ref{le-est-a}, the oscillation errors are estimated as
\begin{equation}\notag
\begin{split}
\|R_{\mathrm{osc,1}}\|_{L^1_tL^\eta_x}\lesssim&\sum_{k\in\Lambda}\|g_{(\tau)}^2\|_{L^1_t}\||\nabla|^{-1}\mathbb P_{\neq 0}\mathbb P_{\geq \lambda_{q+1}r/2}\left((W_{(k)}\otimes W_{(k)})\nabla a_{(k)}^2 \right)\|_{C_tL^\eta_x}\\
\lesssim&\sum_{k\in\Lambda}\|g_{(\tau)}\|_{L^2_t}^2\lambda_{q+1}^{-1}r^{-1}\|\nabla a_{(k)}\|_{C_{t,x}}\|\phi_{(k)}\|_{L^1_tL^{2\eta}_x}^2\|\psi_{(k)}\|_{L^1_tL^{2\eta}_x}^2\\
\lesssim&\ \lambda_{q+1}^{-1}r^{-1}r^{\frac2\eta-2}\ell^{\frac1\eta-1}\\
\lesssim&\ \lambda_{q+1}^{-1}r^{\frac2\eta-3}\ell^{\frac1\eta-1},
\end{split}
\end{equation}
\begin{equation}\notag
\begin{split}
\|R_{\mathrm{osc,2}}\|_{L^1_tL^\eta_x}\lesssim&\ \sigma^{-1}\sum_{k\in\Lambda}\|h_{(\tau)}\|_{C_t}\left( \|a_{(k)}\|_{C_{t,x}}\|\nabla a_{(k)}\|_{C^1_{t,x}}+\| a_{(k)}\|_{C^1_{t,x}}^2\right)\\
\lesssim&\ \sigma^{-1},
\end{split}
\end{equation}
and
\begin{equation}\notag
\begin{split}
\|R_{\mathrm{osc,3}}\|_{L^1_tL^\eta_x}\lesssim&\ \mu^{-1}\sum_{k\in\Lambda} \|\mathcal R\mathbb P_{H}\mathbb P_{\neq 0}\left( \partial_t(a_{(k)}^2g_{(\tau)}^2)\psi_{(k)}^2\phi_{(k)}^2k\right)\|_{L^1_tL^\eta_x}\\
\lesssim&\ \mu^{-1}\sum_{k\in\Lambda} \| \partial_t(a_{(k)}^2g_{(\tau)}^2)\|_{L^1_tC_x}\|\psi_{(k)}^2\phi_{(k)}^2\|_{C_tL^\eta_x}\\
\lesssim&\ \mu^{-1}\sum_{k\in\Lambda}\left(\| \partial_ta_{(k)}^2\|_{C_{t,x}}\|g_{(\tau)}^2\|_{L^1_t} +\|a_{(k)}^2\|_{C_{t,x}}\|\partial_tg_{(\tau)}^2\|_{L^1_t}\right) \|\phi_{(k)}^2\|_{C_tL^{\eta}_x}\|\psi_{(k)}^2\|_{C_tL^{\eta}_x}\\
\lesssim&\ \mu^{-1}(1+\sigma\tau)r^{\frac2\eta-2}\ell^{\frac1{\eta}-1}
\end{split}
\end{equation}
where we used the fact $\phi$ and $\psi$ depend on different components of $x$. 

In the end, applying Lemma \ref{le-est-wv} we have the estimate for the corrector error,
\begin{equation}\notag
\begin{split}
\|R_{\mathrm{cor}}\|_{L^1_tL^\eta_x}\lesssim & \|(w_{q+1}^c+w_{q+1}^t)\otimes w_{q+1}+w_{q+1}^p\otimes (w_{q+1}^c+w_{q+1}^t)\|_{L^1_tL^\eta_x}\\
\lesssim& \left(\|w_{q+1}^c\|_{L^2_tL^{\frac{2\eta}{2-\eta}}_x}+\|w_{q+1}^t\|_{L^2_tL^{\frac{2\eta}{2-\eta}}_x} \right)\left(\|w_{q+1}^p\|_{L^2_tL^2_x} +\|w_{q+1}\|_{L^2_tL^2_x}\right)\\
\lesssim&\left(r^{\frac2\eta-1}\ell^{\frac1\eta-2}+ \lambda_{q+1}\mu^{-1}r^{\frac2\eta-3}\ell^{\frac1\eta-\frac32}\tau^{\frac12}\right)\delta_{q+1}^{\frac12}.
\end{split}
\end{equation}

\subsection{Choice of parameters}
It is time to verify \eqref{induct-R} and \eqref{iter-1}-\eqref{iter-3}, i.e.
\[\|R_{q+1}\|_{L^1_tL^\eta_x}\leq \delta_{q+2},   \ \ \|w_{q+1}\|_{L^2_{t,x}}\leq \delta_{q+1}^{\frac12}, \ \ \|w_{q+1}\|_{L^1_t L^2_x}\leq \delta_{q+1}^{\frac12}, \ \ \|w_{q+1}\|_{L^\gamma_t W^{1,p}_x}\leq \delta_{q+2}^{\frac12}. \]
Thus we require
\begin{equation}\notag
\begin{split}
C\tau^{-\frac12}r^{\frac2\eta-1}\ell^{\frac1\eta-\frac12}\lambda_{q+1}^{-2}\left(1+\sigma\tau+ \frac{\lambda_{q+1}r\mu}{\ell}\right)\leq&\ \delta_{q+2},\\
C\tau^{-\frac12}r^{\frac2\eta-1}\ell^{\frac1\eta-\frac12}\lambda_{q+1}^{2\alpha-2}\leq&\ \delta_{q+2},\\
C\sigma^{-1}\leq &\ \delta_{q+2},\\
C\mu^{-1}r^{\frac2\eta-2}\ell^{\frac1\eta-1}\lambda_{q+1}^{2\alpha-1}\leq&\ \delta_{q+2},\\
C\lambda_q^5 \left(\tau^{-\frac12}r^{\frac2\eta-1}\ell^{\frac1\eta-\frac12}+\tau^{-\frac12}r^{\frac{2}{\eta}}\ell^{\frac1\eta-\frac32}+\lambda_{q+1}\mu^{-1}r^{\frac2\eta-2}\ell^{\frac1\eta-1}\right)\leq&\ \delta_{q+2},\\
C\left(\lambda_{q+1}^{-1}r^{\frac2\eta-3}\ell^{\frac1\eta-1}+\mu^{-1}(1+\sigma\tau)r^{\frac2\eta-2}\ell^{\frac1{\eta}-1} \right)\leq&\ \delta_{q+2},\\
C\left(r^{\frac2\eta-1}\ell^{\frac1\eta-2}+ \lambda_{q+1}\mu^{-1}r^{\frac2\eta-3}\ell^{\frac1\eta-\frac32}\tau^{\frac12}\right)\delta_{q+1}^{\frac12}\leq&\ \delta_{q+2},\\
C\left(r\ell^{-1}+\mu^{-1}\lambda_{q+1}r^{-1}\ell^{-\frac12}\tau^{\frac12}\right)\leq &\ \delta_{q+1}^{\frac12},\\
C\left(\tau^{-\frac12}+ r\ell^{-1}\tau^{-\frac12}+\mu^{-1}\lambda_{q+1} r^{-1}\ell^{-\frac12}\right)\leq &\ \delta_{q+1}^{\frac12},\\
C\left(\lambda_{q+1} r^{\frac2p-1}\ell^{\frac1p-\frac12}\tau^{\frac12-\frac1\gamma}+\lambda_{q+1} r^{\frac2p}\ell^{\frac1p-\frac32}\tau^{\frac12-\frac1\gamma}+\mu^{-1}\lambda_{q+1}^2 r^{\frac2p-2}\ell^{\frac1p-1}\tau^{1-\frac1\gamma} \right)\leq &\ \delta_{q+2}^{\frac12}
\end{split}
\end{equation}
for some constant $C>0$. Recall
\[r=\lambda_{q+1}^{n_1}, \ \ \ell=\lambda_{q+1}^{n_2}, \ \ \mu=\lambda_{q+1}^{n_3},  \ \ \tau=\lambda_{q+1}^{n_4}, \ \ \sigma=\lambda_{q+1}^{2\ve}, \ \ \delta_{q+2}=\lambda_{q+1}^{-2b\beta}\]
for some constants $n_1<n_2<0$ and $n_3,n_4>0$.
Thus the conditions above will be satisfied provided 
\begin{equation}\label{para-2}
\begin{split}
(2/\eta-1)n_1+(1/\eta-\frac12)n_2+\frac12n_4+2\ve-2<&-2b\beta,\\
2/\eta n_1+(1/\eta-\frac32)n_2+n_3-\frac12n_4-1<&-2b\beta,\\
(2/\eta-1)n_1+(1/\eta-\frac12)n_2-\frac12n_4+2\alpha-2<&-2b\beta,\\
(2/\eta-2)n_1+(1/\eta-1)n_2-n_3+2\alpha-1<&-2b\beta,\\
(2/\eta-1)n_1+(1/\eta-\frac12)n_2-\frac12n_4+5/b<&-2b\beta,\\
(2/\eta-2)n_1+(1/\eta-1)n_2-n_3+1+5/b<&-2b\beta,\\
(2/\eta-3)n_1+(1/\eta-1)n_2-1<&-2b\beta,\\
(2/\eta-2)n_1+(1/\eta-1)n_2-n_3+n_4+2\ve<&-2b\beta,\\
(2/\eta-1)n_1+(1/\eta-2)n_2-\beta<&-2b\beta,\\
(2/\eta-3)n_1+(1/\eta-\frac32)n_2-n_3+\frac12n_4+1-\beta<&-2b\beta,\\
-2\ve <&-2b\beta,\\
\end{split}
\end{equation}
\begin{equation}\label{para-3}
\begin{split}
n_1-n_2<&-\beta,\\
-n_1-\frac12n_2-n_3+\frac12n_4+1<&-\beta,\\
-\frac12n_4<-\beta, \ \ n_1-n_2-\frac12n_4<-\beta, \ \ -n_1-\frac12n_2-n_3+1<&-\beta,
\end{split}
\end{equation}
\begin{equation}\label{para-4}
\begin{split}
(\frac2p-1)n_1+(\frac1p-\frac12)n_2+(\frac12-\frac1\gamma)n_4+1<&-b\beta,\\
(\frac2p-2)n_1+(\frac1p-1)n_2-n_3+(1-\frac1\gamma)n_4+2<&-b\beta.
\end{split}
\end{equation}
Since $\eta$ can be chosen as close as to $1$, we take $\eta=1$ in the first set of conditions (\ref{para-2}) for brevity and obtain for $\ve>2b\beta$
\begin{equation}\label{para-5}
\begin{split}
n_1+\frac12n_2+\frac12n_4+2\ve-2<&-2b\beta,\\
2 n_1-\frac12n_2+n_3-\frac12n_4-1<&-2b\beta,\\
n_1+\frac12n_2-\frac12n_4+2\alpha-2<&-2b\beta,\\
-n_3+2\alpha-1<&-2b\beta,\\
n_1+\frac12n_2-\frac12n_4+5/b<&-2b\beta,\\
-n_3+1+5/b<&-2b\beta,\\
n_1-1<&-2b\beta,\\
-n_3+n_4+2\ve<&-2b\beta,\\
n_1-n_2-\beta<&-2b\beta,\\
-n_1-\frac12n_2-n_3+\frac12n_4+1-\beta<&-2b\beta.
\end{split}
\end{equation}
Analyzing the inequalities in (\ref{para-3}) and (\ref{para-5}), we first choose 
\begin{equation}\label{para-6}
n_1=-1+2\ve, \ \ n_2=-1+4\ve.
\end{equation}
Note that by such choice, we have from Lemma \ref{le-est-psi}
\[\|W_{(k)}\|_{C_tL_x^p}\lesssim r^{\frac2p-1}\ell^{\frac1p-\frac12}\lesssim \lambda_{q+1}^{(3-8\ve)(\frac12-\frac1p)} \]
which indicates almost full spatial concentration (extreme intermittency). 

With choice \eqref{para-6}, taking $\gamma=\infty$ in (\ref{para-4}) gives
\begin{equation}\label{cond-p}
p<\frac{3-8\ve}{\frac52+\frac12n_4+b\beta-4\ve}, \ \ \ p<\frac{3-8\ve}{5-n_3+n_4+b\beta-8\ve}.
\end{equation}
In view of \eqref{cond-p}, to maximize $p$, we need to choose $n_3$ as large as possible and $n_4$ as small as possible. Since $\mu=\lambda_{q+1}^{n_3}$ and $\tau=\lambda_{q+1}^{n_4}$ are respectively the temporal oscillation and concentration parameters, the above observation says that the constructed solutions are highly oscillatory in time with minimal temporal concentration. In the case of $1\leq \alpha<\frac74$, we take 
\begin{equation}\label{para-7}
n_3=\frac52+2\ve, \ \ n_4=10\ve. 
\end{equation}
One can verify that with $n_1,n_2,n_3$ and $n_4$ chosen in \eqref{para-6} and \eqref{para-7}, all the conditions in \eqref{para-5} are satisfied. It then follows from \eqref{cond-p} that $p<\frac65$. While in the case $\alpha\in[\frac74, 3-6\ve+b\beta-\beta)$ we choose
\begin{equation}\label{para-8}
n_3=2\alpha-1+2\ve, \ \ n_4=4\alpha-7+10\ve
\end{equation}
which together with \eqref{para-6} makes the inequalities of \eqref{para-5} valid. Again, it follows from \eqref{cond-p} that
\[p<\frac{3-8\ve}{2\alpha-1+\ve+b\beta}\]
which indicates for sufficiently small $\ve>2b\beta>\beta>0$ that $p<\frac{3}{2\alpha-1}$. 

The inductive estimates \eqref{induct-A}-\eqref{induct-R} and the rest of the conclusions of Proposition \ref{prop} can be obtained in an analogous way as in the end of Section \ref{sec-case-1}.

\begin{rem}\label{rem-case2}
Again, if we consider the electron MHD \eqref{emhd} without resistivity $(-\Delta)^\alpha B$, the conditions 
\begin{equation}\notag
\begin{split}
n_1+\frac12n_2-\frac12n_4+2\alpha-2<&-2b\beta,\\
-n_3+2\alpha-1<&-2b\beta,\\
\end{split}
\end{equation}
from \eqref{para-5} should be removed. In this case, the parameter choices in \eqref{para-6} and \eqref{para-7} give the optimal value of $p$ satisfying \eqref{cond-p}: $p<\frac65$.
\end{rem}


\bigskip

\appendix

\section{Technical Lemmas}

We recall the geometric lemma introduced by Nash \cite{Nash}.
\begin{lem} \label{le-geo}
Let $B_{\frac12}(\mathrm{Id})$ be the ball of radius $\frac12$ centered at the identity in the space of $3\times 3$ symmetric matrices. There exists a finite set $\Lambda\subset \mathbb S^2\cap \mathbb Q^3$ consisting of vectors $k$ with associated orthonormal bases $\{k,k_1,k_2\}$ and smooth functions $\gamma_{(k)}: B_{\frac12}(\mathrm{Id})\to \mathbb R$ such that
\[R=\sum_{k\in\Lambda}\gamma_{(k)}^2(R)k\otimes k, \ \ \ R\in B_{\frac12}(\mathrm{Id}). \]
\end{lem}



The following $L^p$ boundedness of Calder\'on-Zygmund operators will be used for estimates involving the inverse divergence operator $\mathcal R$.
\begin{lem}\label{le-CZ}
Calder\'on-Zygmund operators are bounded on the space $L^p$ of zero-mean functions for $p\in(1,\infty)$.
\end{lem}



\bigskip



\begin{thebibliography}{XX}

\bibitem{Bis1}
D. Biskamp.
\newblock {\em Magnetic reconnection in plasmas}.
\newblock Cambridge University Press, 2000.

\bibitem{BCV}
T. Buckmaster, M. Colombo, and V. Vicol.
\newblock {\em Wild solutions of the Navier-Stokes equations whose singular sets in time have Hausdorff dimension strictly less than 1}.
\newblock Journal of the European Mathematics Society, 24(9):3333--3378, 2021.


\bibitem{BV}
T. Buckmaster, and V. Vicol.
\newblock {\em Nonuniqueness of weak solutions to the Navier-Stokes equation}.
\newblock Ann. of Math., (2) 189(1): 101--144, 2019.

\bibitem{CL}
D. Chae and J. Lee.
\newblock {\em On the blow-up criterion and small data global existence for the Hall-magnetohydrodynamics}.
\newblock J. Differential Equations, 256: 3835--3858, 2014.

\bibitem{CD-nse}
A. Cheskidov and M. Dai.
\newblock {\em Kolmogorov's dissipation number and the number of degrees of freedom for the 3D Navier-Stokes equations}.
\newblock Proceedings of the Royal Society of Edinburg, Section A, Vol. 149, Issue 2:429--446, 2019.

\bibitem{ChL2}
A. Cheskidov and X. Luo.
\newblock {\em $L^2$-critical nonuniqueness for the 2D Navier-Stokes equations}.
\newblock Annals of PDE, Vol. 9, 13, 2023.

\bibitem{ChL}
A. Cheskidov and X. Luo.
\newblock {\em Sharp uniqueness for the Navier-Stokes equations}.
\newblock Invent. Math., https://doi.org/10.1007/s00222-022-01116-x, 2022.

\bibitem{Dai}
M. Dai.
\newblock {\em Non-uniqueness of weak solutions in Leray-Hopf space for the 3D Hall-MHD system}.
\newblock SIAM Journal of Mathematical Analysis, 53(5): 5979--6016, 2021.

\bibitem{DSz}
S. Daneri and L. Sz\'ekelyhidi Jr.
\newblock {\em Non-Uniqueness and H-Principle for H\"older continuous weak solutions of the Euler equations}.
\newblock Archive for Rational Mechanics and Analysis, 224:471--514, 2017.


\bibitem{GR}
V. Giri and R. O. Radu.
\newblock {\em The 2D Onsager conjecture: a Newton-Nash iteration}.
\newblock arXiv:2305.18105, 2023.

\bibitem{Gor}
M. Gorini.
\newblock {\em Density of weak solutions of the fractional Navier-Stokes equations in the smooth divergence-free vector fields}.
\newblock arXiv:2312.00839, 2023.





\bibitem{LQZZ}
Y. Li, P. Qu, Z. Zeng and D. Zhang.
\newblock {\em Sharp non-uniqueness for the 3D hyperdissipative Navier-Stokes equations: above the Lions exponent}.
\newblock arXiv:2205.10260, 2022.

\bibitem{LZZ}
Y. Li, Z. Zeng and D. Zhang.
\newblock {\em Sharp non-uniqueness of weak solutions to 3D magnetohydrodynamic equations}.
\newblock arXiv:2208.00624, 2022.


\bibitem{Nash}
J. Nash.
\newblock {\em $C^1$ isometric imbeddings}.
\newblock Ann. Math., 2(60): 383--396, 1954.

\bibitem{Ye}
Z. Ye.
\newblock {\em Regularity criteria and small data global existence to the generalized viscous Hall-magnetohydrodynamics}.
\newblock Comput. Math. Appl., 70:2137--2154, 2015.

\end{thebibliography}
\end{document}